\documentclass[11pt,reqno]{amsart}
\usepackage{mathrsfs}
\usepackage[american]{babel}
\usepackage{todonotes}
\usepackage{comment}
\usepackage{amsmath}
\usepackage{dsfont,mathtools,amssymb}
\usepackage{appendix}
\usepackage{color}
\usepackage{nicematrix}
\usepackage{tikz}
\usetikzlibrary{decorations.pathreplacing,fit,backgrounds,positioning}
\usepackage[hidelinks]{hyperref}
\usepackage[shortlabels]{enumitem}
\allowdisplaybreaks

\usepackage[margin=1.2in]{geometry}
\usepackage{array,longtable}
\usepackage[normalem]{ulem}

\theoremstyle{plain}
\newtheorem{theorem}{Theorem}[section]
\newtheorem{proposition}[theorem]{Proposition}
\newtheorem{lemma}[theorem]{Lemma}
\newtheorem{corollary}[theorem]{Corollary}
\theoremstyle{remark}
\newtheorem{remark}[theorem]{Remark}
\theoremstyle{definition}
\newtheorem{definition}[theorem]{Definition}

\newtheorem{conjecture}[theorem]{Conjecture}

\numberwithin{equation}{section}
\numberwithin{figure}{section}

\renewcommand{\complement}{c}
\newcommand{\E}{\mathds{E}}
\renewcommand{\P}{\mathds{P}}

\newcommand{\X}{\mathbf{X}}
\newcommand{\R}{\mathbb{R}}
\newcommand{\T}{\mathbb{T}}

\newcommand{\Z}{\mathbb{Z}}
\newcommand{\N}{\mathbb{N}}
\newcommand{\dd}{{\rm d}}
\newcommand{\fstop}{\; \text{.}}
\newcommand{\comma}{\; \text{,}\;\;}

\newcommand{\tonde}[1]{\left(#1\right)}
\newcommand{\quadre}[1]{\left[#1\right]}
\newcommand{\ttonde}[1]{\big(#1\big)}
\newcommand{\tttonde}[1]{(#1)}
\newcommand{\abs}[1]{\left\lvert#1\right\rvert}
\newcommand{\tabs}[1]{\big\lvert#1\big\rvert}

\newcommand{\emparg}{\,\cdot\,}
\newcommand{\emp}{\varnothing}
\newcommand{\eqdef}{\coloneqq}

\newcommand{\car}{\mathds{1}}

\newcommand{\ttnorm}[1]{\|#1\|}
 
\newcommand{\eps}{\varepsilon} 
\newcommand{\var}{{\rm Var}}

\newcommand{\set}[1]{\left\{#1\right\}}							

\newcommand{\cC}{\ensuremath{\mathcal C}}

\newcommand{\cF}{\ensuremath{\mathcal F}} 
\newcommand{\cG}{\ensuremath{\mathcal G}} 
 
\newcommand{\cI}{\ensuremath{\mathcal I}} 
\newcommand{\cJ}{\ensuremath{\mathcal J}} 
 
\newcommand{\cL}{\ensuremath{\mathcal L}}

\newcommand{\cP}{\ensuremath{\mathcal P}}

\newcommand{\cS}{\ensuremath{\mathcal S}} 
\newcommand{\cT}{\ensuremath{\mathcal T}}

\newcommand{\cW}{\ensuremath{\mathcal W}} 
 
\newcommand{\cY}{\ensuremath{\mathcal Y}}

\newcommand{\DEP}{{\rm \scriptscriptstyle DEP}}
\newcommand{\Toom}{{\rm \scriptscriptstyle Toom}}

\allowdisplaybreaks

\begin{document} 
	
\title[Hydrodynamic limit of DEP]{\small Hydrodynamic limit of the directed exclusion process}

\author{Ellen Saada}
\author{Federico Sau}
\author{Assaf Shapira}

\makeatletter
\def\@setauthors{%
	\begingroup
	\trivlist
	\centering\footnotesize
	\item\relax
	\vskip 2.0em 
	{\normalsize
	Ellen Saada$^{a}$\quad
Federico Sau$^{b}$\quad
	Assaf Shapira$^{c}$}\par
	\vskip 2.0em
	$^{a}$\mbox{\small CNRS, UMR 8145, MAP5, Universit\'e Paris Cit\'e}
	\\ \texttt{ellen.saada@mi.parisdescartes.fr}\par
	\vskip 0.8em
	$^{b}$\mbox{\small	Dipartimento di Matematica \textquotedblleft F. Enriques\textquotedblright,
	Università degli Studi di Milano}
	\\ \texttt{federico.sau@unimi.it}\par\vskip 0.8em
	$^{c}$\mbox{\small Universit\'e Paris Cit\'e, CNRS, MAP5}
	\\ \texttt{assaf.shapira@math.cnrs.fr}\par
	\endtrivlist
	\endgroup
}
\makeatother

\thispagestyle{empty}

		\begin{abstract}
	We derive the Euler (hyperbolic) hydrodynamic limit for the directed exclusion process (${\rm DEP}$), a one-dimensional conservative interacting particle system that preserves particle–hole symmetry while breaking left–right symmetry. The proof  relies on an explicit multi-process coupling, which guarantees a strong form of attractiveness and macroscopic stability for the particle system. Further open questions about ${\rm DEP}$ are briefly discussed.
	\end{abstract}

\maketitle

	\begin{center}{\textit{Dedicated to Claudio Landim on the occasion of his 60th birthday}}
\end{center}

		\thispagestyle{empty}	

\section{Introduction}\label{sec:intro}

The (symmetric) directed exclusion process (${\rm DEP}$) is an interacting
particle system studied in the physics literature as a simple example
of a model belonging to the \emph{advected Edwards--Wilkinson} 
universality class \cite{devilliard_spohn_universality_1992,binder1994scaling,shapira_wiese_anchored_2023}.
This universality class consists of models preserving the particle-hole
symmetry, but breaking the directional left-right symmetry, and includes
also the very well studied Toom interface model 
\cite{derrida_dynamics_1991,crawford2016invariance,crawford_kozma_toom_2020}.

For ${\rm DEP}$, particles are placed on $\Z$ with exclusion (i.e., at most
one particle per site). Then, each of the following transitions occurs
with rate $1$: 
\begin{enumerate}
	\item A particle at $x$ with a neighboring empty site $y=x\pm1$ jumps
	to $y$. 
	\item A particle at $x$ with a particle to its right at $x+1$ and an empty
	site at $x+2$ jumps to $x+2$. 
	\item A particle at $x$ with two empty sites to its left, at $x-1$ and
	$x-2$, jumps to $x-2$. 
\end{enumerate}
Let $\eta\in\{0,1\}^{\Z}$ be the particle configuration, so the
hole configuration is  $\widetilde{\eta}=\mathbf1-\eta$ 
(where $\mathbf1$ is the full configuration, that is, $\mathbf1(x)=1$ for all $x\in\Z$). 
One can verify that both
$\eta$ and $\widetilde{\eta}$ evolve as the same process,
i.e., the law of ${\rm DEP}$ is invariant under particle-hole symmetry.

It is instructive to compare this model with the most renowned \emph{symmetric
simple	exclusion process} (SSEP), where only the first of the three transitions
above occurs. The SSEP has the same particle-hole symmetry, in addition
to a directional symmetry: the configuration $\eta_{\text{SSEP}}$
of SSEP evolves according to the same law as the reflected configuration
$\left(\eta_{\text{SSEP}}(-x)\right)_{x\in\Z}$. A natural way to
break the directional symmetry of SSEP is to give different rates
to jumps to the right and to the left, obtaining the \emph{asymmetric
	simple exclusion process}. This, however, will also break the particle-hole
symmetry with it. 

In ${\rm DEP}$, just like in ${\rm SSEP}$, any Bernoulli product measure $\nu_\rho$, $\rho \in [0,1]$, 
is stationary (see Section \ref{sec:model+properties}). However,  unlike SSEP, 
directional symmetry is broken, although particle-hole
symmetry is preserved. This last property  is the reason why one expects ${\rm DEP}$'s equilibrium
density fluctuation field  at criticality (i.e., around particle
density $\rho=1/2$,  see later on)  to behave in the limit according to the
\emph{advected Edwards--Wilkinson equation} on $\R$ (see also Section
\ref{sec:adEW}): 
\begin{equation}\label{eq:aEW}
	\partial_{t}\cY=\partial_x\tttonde{-\mu \cY+D\partial_x \cY+\sigma \cW}\ ,
\end{equation}
where $\cW$ is a space-time white noise, and $\mu,D$ and $\sigma$ positive
coefficients. We note that the directional symmetry breaking allows
for a non-zero advection term $\mu\partial_x \cY$, while the particle-hole
symmetry forbids a KPZ-type term $\cY\partial_x \cY$.

\subsection{Hydrodynamics}\label{subsec:hydro}
The purpose of this paper is to make a first step into the
analysis of large scale limits of ${\rm DEP}$, by proving a hydrodynamic limit
\cite{de_masi_mathematical_1991,kipnis_scaling_1999} for the model:
provided that the empirical density field at the initial time approximates
a profile $u_{0}:\R\to[0,1]$, then, under a hyperbolic space-time
scaling, ${\rm DEP}$ approximates, at any later time $t>0$, the profile $u(\emparg,t):\R\to[0,1]$,
suitable solution to $u(\emparg,0)=u_0$ and
\begin{equation}\label{eq:HDL}
	\partial_{t}u  +\partial_x G_{\DEP}(u)=0\comma\quad\text{with }G_{\DEP}(u)\eqdef2u\tonde{1-u}\tonde{2u-1}\fstop
\end{equation}
The precise result is the content of Theorem \ref{th:HDL}.
Its proof is based on  a constructive method developed in 
\cite{bahadoran_constructive_2002,bahadoran_euler_2006,bgrs3,bgrs5},
well suited for hydrodynamic limits of one-dimensional conservative attractive particle systems
under a hyperbolic space-time scaling. Its first step is 
to derive ``Riemann hydrodynamics'' (i.e., for the case where
$u_0$ is a one-step function); then, to prove general (Cauchy) hydrodynamics 
through an approximation scheme inspired by Glimm's scheme for conservation laws. 
The latter requires the following essential properties of the dynamics
 (defined and derived in Section \ref{sec:coupling}): 
(a) monotonicity of an arbitrary number of copies of the system; (b) 
macroscopic stability; (c) finite propagation property.
Our general strategy builds upon the construction
of suitable couplings which  guarantee these properties. 
\smallskip

In the remainder of this section, we discuss some further open questions 
for ${\rm DEP}$.

\subsection{Fluctuations at criticality}\label{sec:adEW}

One can see from the hydrodynamic equation \eqref{eq:HDL} that $\rho=1/2$
is a critical density of ${\rm DEP}$: when $\rho<1/2$ there is an
overall particle current to the left  ($G_{\DEP}(\rho)<0)$,  
while for $\rho>1/2$
the current is to the right.

Consider a (small) scale parameter $\eps\in(0,1)$, and a scaling function 
$s_\eps$ that will associate to a macroscopic time $t$ the microscopic time 
$s_\eps(t)$; for diffusive scaling, for example, $s_\eps(t) = \eps^{-2} t$.
 Starting ${\rm DEP}$ from the (critical) stationary Bernoulli product measure 
 $\nu_{1/2}$,
we consider the fluctuation field at scale $\varepsilon$ associated
with the configuration $\eta_{s_\eps(t)}$,  acting on  test functions $f:\R\to\R$
as 
\begin{equation}\label{eq:field-tilde}
	\cY_{t}^\eps(f)=\varepsilon^{1/2}\sum_{x\in\Z}
	f(\varepsilon x-\varepsilon \mu s_\eps(t))\left(\eta_{s_\eps(t)}(x)
	-\frac{1}{2}\right)\ ,
\end{equation}
where $\mu=1$ is the conjectured coefficient of the advection term
in the equation \eqref{eq:aEW}.

Let us look at the time evolution of $\cY_{t}^{\varepsilon}$ 
more closely. Letting $\cL$ denote the infinitesimal generator of 
${\rm DEP}$ (cf.\ \eqref{eq:generator}), a short calculation 
(see \eqref{eq:Leta} or \cite{shapira_wiese_anchored_2023})  yields,
 for all $x\in \Z$,
\begin{align*}
	\cL\eta(x)&= \car_{\set{\eta(x-2)=\eta(x-1)\neq \eta(x)}}\tonde{\eta(x-2)-\eta(x)}+\tonde{\eta(x-1)-\eta(x)}\\[.25cm]
	&\qquad+ \tonde{\eta(x+1)-\eta(x)}+\car_{\set{\eta(x)=\eta(x+1)\neq \eta(x+2)}}\tonde{\eta(x+2)-\eta(x)}\\
	&= 2\car_{\set{\eta(x-2)=\eta(x-1)\neq \eta(x)}}\tonde{\eta(x-1)-\frac12}+\tonde{\eta(x+1)+\eta(x-1)-2\eta(x)}\\
	&\qquad-2\car_{\set{\eta(x)=\eta(x+1)\neq \eta(x+2)}}\tonde{\eta(x+1)-\frac12}\comma
\end{align*}
which, by an integration by parts and the Taylor expansions 
\begin{equation*}
	f(\eps (x+1)-\eps \mu s_\eps(t))-f(\eps (x-1)-\eps \mu s_\eps(t))= 2\eps f'(\eps x-\eps \mu s_\eps(t))+O(\eps^3)\comma
\end{equation*}
\begin{equation*}
f(\eps(x+1)-\eps \mu s_\eps(t))+ f(\eps (x-1)-\eps \mu s_\eps(t))-2f(\eps x-\eps \mu s_\eps(t))= \eps^2 f''(\eps x-\eps \mu s_\eps(t))+O(\eps^4)\comma
\end{equation*}
	further gives, up to lower order terms arising from Taylor expansions of $f$, 
\begin{align}\label{eq:pre-replacement}
	\begin{aligned}
		&\cL \cY_{t}^{\varepsilon}(f) 
	 =\varepsilon^{1/2}\,\varepsilon^{2}
	\sum_{x\in\Z}f''(\varepsilon x-\varepsilon \mu s_\eps(t))
	\left(\eta_{s_\eps(t)}(x)-\frac{1}{2}\right)\\
		& +\varepsilon^{1/2}\,4\varepsilon\sum_{x\in\Z}
		\car_{\{\eta_{s_\eps(t)}(x-1)=\eta_{s_\eps(t)}(x)\neq\eta_{s_\eps(t)}(x+1)\}}\, 
		f'(\varepsilon x-\varepsilon \mu s_\eps(t))
		\left(\eta_{s_\eps(t)}(x)-\frac{1}{2}\right)\fstop
	\end{aligned}
\end{align} 
If we had some type of replacement lemma with respect to $\nu_{1/2}$,  since
\begin{equation*}
	\nu_{1/2}(\car_{\{\eta(x-1)=\eta(x)\neq\eta(x+1)\}}\mid \eta(x))=\frac14\comma
\end{equation*} 
we could rewrite the last
 line in \eqref{eq:pre-replacement}   as 
\begin{equation}\label{eq:replacement}
	\varepsilon^{1/2}\,\varepsilon\sum_{x\in\Z}
	f'(\varepsilon x-\varepsilon \mu s_\eps(t))
	\left(\eta_{s_\eps(t)}(x)-\frac{1}{2}\right)\fstop
\end{equation}
 As $\partial_t f(\eps x-\eps \mu s_\eps(t))=-\eps \mu s_\eps'(t) f'(\eps x-\eps \mu s_\eps(t))$
 imposing a diffusive space-time scaling 
(i.e., setting $s_\eps(t)=\eps^{-2} t$, so that $s_\eps'(t)=\eps^{-2}$),   the choice $\mu=1$  would cancel the time
derivative    
\begin{align*}
\partial_t \cY_t^\eps(f)&=-\eps^{1/2} \eps \mu s_\eps'(t) \sum_{x\in \Z}f'(\eps x-\eps \mu s_\eps(t))\tonde{\eta_{s_\eps(t)}(x)-\frac12}
\comma
\end{align*}   of $\cY_{t}^{\varepsilon}$, yielding
\begin{equation*}
	\frac{{\rm d}}{{\rm d}t}\E\big[\cY_{t}^{\varepsilon}(f)\big]
	 =	 \tttonde{\partial_t + \eps^{-2}\cL}\cY_t^\eps(f)  \approx \E\big[\cY_{t}^{\varepsilon}(f'')\big]\fstop
\end{equation*}
 Further, we may expect the quadratic variation to scale as for SSEP: 
 for all $t\ge 0$,
\begin{equation*}
	\var(\cY_{t+\dd t}^{\varepsilon}(f)\mid \eta_{t})  
	\propto  \eps^{-2}\eps\sum_{x\in \Z}(\eps f'(\eps x-\eps^{-1} \mu t))^2\,\dd t 
	\approx 
	 \ttnorm{f'}_{L^2(\R)}^2\dd t\fstop
\end{equation*}
If integrated over time, the right-hand side describes the variance of  
$\int \partial_xf\,\cW$. Hence,  the above heuristic arguments seem to suggest 
that, for small $\eps\in (0,1)$, the field in \eqref{eq:field-tilde} 
is an approximate solution to the (non-advected) Edwards-Wilkinson equation: 
\begin{equation*}
	\partial_t\cY_{t}^{\varepsilon}
	\approx\partial_x\big(D\partial_x \cY_{t}^{\varepsilon}
	+\sigma \cW\big)\ .
\end{equation*}
This is the way we interpret \eqref{eq:aEW}: first change to a frame
of reference that moves with microscopic speed $\varepsilon\mu$; then, under
diffusive scaling, the field converges to a solution of  
\begin{equation}\label{eq:EW}
	\partial_t\cY_t=\partial_x\big(D\partial_x \cY_t+\sigma \cW\big)\fstop
\end{equation}
We stress that if one scales diffusively without  adjusting a
frame of reference (i.e., one sets $\mu\neq 1$ in \eqref{eq:field-tilde}), the speed in diffusive time diverges as $\varepsilon^{-1}(\mu-1)$,
meaning that advection is at a much faster scale than diffusion.

Unfortunately, replacing $\car_{\set{\eta(x-1)=\eta(x)\neq\eta(x+1)}}$
by its expectation (in passing from \eqref{eq:pre-replacement} to
\eqref{eq:replacement}) is not allowed: this is reflected in the
fact that an additional $\cY^{2}\partial_x \cY$ term in \eqref{eq:EW} is
\emph{marginally relevant}, and cannot be simply neglected. It is
generically expected in such cases that the limiting field is still
described by \eqref{eq:EW}, but with logarithmic corrections to the scaling;  
in this case, this correction is conjectured to be  \cite{vanbeijeren_kutner_spohn_excess_1985,paczuski1992fluctuations,devilliard_spohn_universality_1992,spohn_nonlinear_2014,cannizzaro_erhard_toninelli_stationary_2023}
\begin{equation}\label{eq:t-eps-s}
	t=\varepsilon^{2}s\log(s)^{1/2}\comma
\end{equation}
that is, the scaling function $s_\eps(t)\ge1$, $ t\ge 0$, is given by  the  
solution of this equation. We note that, for fixed $t>0$,  
$s_\eps(t)\sim  \frac{t\eps^{-2}}{\sqrt 2}|\log \eps|^{-1/2}$ as $\eps\to 0$. 

\begin{conjecture}\label{conj:fluctuations}
	 Let $s_\eps(t)$ be the solution to \eqref{eq:t-eps-s}. Then,   
	the field $(\cY_t^\eps)_{t\ge 0}$ given, for all test functions 
	$f:\R\to \R$ and $t\ge 0$, by
	\begin{equation}\label{eq:field-micro}
		\cY_t^\eps(f)= \eps^{1/2}
		\sum_{x\in \Z}f(\eps x-\eps s_\eps(t))\tonde{\eta_{s_\eps(t)}(x) -\frac12}
	\end{equation} 
	converges as $\eps\to 0$, in the appropriate distributional
	space, to the infinite dimensional Ornstein--Uhlenbeck process described
	by equation \eqref{eq:EW}.
\end{conjecture}

\subsection{Dynamics with boundary}\label{subsec:boundary}

A very interesting variant of ${\rm DEP}$ is on the half-line 
$\N^*\eqdef \{1,2,\ldots\}$, not allowing
particles to jump beyond $0$. In this setting, we expect ${\rm DEP}$ to exhibit
 \emph{self-organized criticality}: that is, starting from any generic
density profile, the system converges,  in an appropriate sense, 
in the long run and under a suitable space-time scaling, to  a critical
state, namely, a full-line invariant measure with flat profile of critical density given, in this case, by $\rho = 1/2$.

This phenomenon can be already guessed from
the hydrodynamic limit equation \eqref{eq:HDL}: by adding a boundary
condition requiring the current to vanish, one gets  $G_{\DEP}(u(0,t))=0$, 
$t>0$, which
implies  $\overline{u}=\lim_{t\to \infty}u(\emparg,t)\equiv 1/2$ 
(provided one can exclude the degenerate cases $\overline{u}\equiv 0$ and $1$).
 Intuitively, at the microscopic level,
excess of particles induces a right current, sending particles to
infinity;  low density induces a left current, sending holes
to infinity. Thus, the system is expected to organize itself in the
critical state, with vanishing current. A rigorous  derivation of this behavior is an open problem, on which we plan to progress in the  future.

It is worth mentioning that self-organized criticality is shown in \cite{crawford_kozma_toom_2020}
for a related model,   also belonging to the advected Edwards-Wilkinson
universality class: the Toom interface model. This model is similar to ${\rm DEP}$, except that it allows
infinite-range jumps: while for ${\rm DEP}$ particles may jump over a single
particle to the right, in Toom's model particles can jump over arbitrarily
many particles to their right for reaching the first empty site to their right. For Toom's model, heuristic arguments suggest an hyperbolic hydrodynamic limit with a flux function given by 
\begin{equation*}
G_{\Toom}(u)\eqdef	\frac{u}{1-u}-\frac{1-u}u\comma\qquad u \in (0,1)\fstop
\end{equation*} On the one hand, infinite range jumps
simplify the analysis, allowing for a coupling where discrepancies
disappear with a fixed rate. On the other hand, one must be careful
in even defining the model, and neither uniform bounds nor finite-propagation properties can be used when deriving
hydrodynamic limits.

An important feature of self-organized criticality is that it allows
us to observe non-trivial scaling exponents, without the need to fine-tune
the model's parameters. Models in the advected Edwards-Wilkinson universality
class on the half-line are expected to be \emph{hyperuniform} \cite{derrida_dynamics_1991,shapira_wiese_anchored_2023}.
That is, the number of particles in the interval $[0,L]$ has variance
much smaller than $L$. The works \cite{derrida_dynamics_1991,paczuski1992fluctuations,devilliard_spohn_universality_1992,shapira_wiese_anchored_2023} propose a more precise prediction for
this universality class, indicating that the variance should scale as $L^{1/2}\log(L)^{1/4}$.

An $L^{1/2}$ scaling can be shown using explicit calculations for
the limiting equation \eqref{eq:aEW} on $\R_+\eqdef (0,\infty)$ with the appropriate
boundary condition \cite{pruessner2004drift,shapira_wiese_anchored_2023}.
We will describe here the intuition leading to this result, with the
additional logarithmic correction. Fix an integer $L> 1$, and recall the field $\cY_t^\eps$ defined in \eqref{eq:field-micro}.  Then, the fluctuation of the number of particles in the interval $[1,L]$ at time $t=0$ reads as	
\begin{equation*}
\cY_0^\eps(\psi_L^\eps)=	\sum_{x=1}^L\tonde{\eta_0(x)-\frac12} \comma\quad \text{with}\ \psi_L^\eps \eqdef \eps^{-1/2}\,\car_{(0,\eps L]}\fstop
\end{equation*}
We will see how this number evolves between time $-t<0$ and time $0$. The scaling invariance
of the Edwards--Wilkinson equation \eqref{eq:EW} and Conjecture \ref{conj:fluctuations} formally imply
\begin{equation}\label{eq:scaling-var-tilde}
	\var\!\left(\cY_{0}^{\varepsilon}(\psi_{L}^\eps)-\cY_{-t}^{\varepsilon}(\psi_{L}^\eps)\right)\propto\varepsilon^{-1}\sqrt{t}\ .
\end{equation}
On the half-line, the above field
should be interpreted with the sum over $x\in\N^*$ rather than $\Z$. Hence, letting $ \cY_t^{\eps,+}$ be the field $\cY_t^\eps$ restricted to the positive half-line, we get,  for $t=\varepsilon^{2}L\log(L)^{1/2}$ (i.e., $s_\eps(t)=L$, thus, yielding, for all $x\ge 1$, 
 $\eps x+\eps s_\eps(t)>\eps L$ and $\psi_L(\eps x+\eps s_\eps(t))=0$),
\begin{align*}
 \cY_{-t}^{\eps,+}(\psi_L^\eps)= \eps^{1/2}\sum_{x\in \N^*} \psi_L^\eps(\eps x+\eps s_\eps(t))\tonde{\eta_{s_\eps(t)}(x)-\frac12} =0\fstop
\end{align*}

During the time interval $(-t,0)$ the field $\cY^{\eps,+}$ evolves in a similar way to $\cY^{\eps}$: for both, at any $t_1\in (-t,0)$, the evolution depends on particles jumping near the position $s_\eps(t_1)$ strictly to the right of the boundary.
As argued in \cite[Section 3.6]{shapira_wiese_anchored_2023}, the effect of the (far away) boundary can be neglected. Hence, the variance scaling in \eqref{eq:scaling-var-tilde} holds true also for $\cY^{\eps,+}$ with $t=\varepsilon^{2}L\log(L)^{1/2}$.  One therefore expects, for  small $\eps \in (0,1)$ and  large $L>0$, 
\begin{equation}\label{eq:var-Y0}
\var\!\ttonde{ \cY_0^{\eps,+}(\psi_L^\eps)}\propto \eps^{-1}\sqrt t = L^{1/2}\log(L)^{1/4}\comma
\end{equation}
thus, motivating the aforementioned particle-number variance scaling and the following conjectural scaling limit.
\begin{conjecture}
	Let $\nu$ be a nontrivial (i.e., { $\nu(\mathbf 0) = \nu (\mathbf 1) = 0$})
	stationary measure of ${\rm DEP}$ on the half-line. Then, provided
	that $\eta\sim\nu$, for any test function $f:\R_{+}\to\R$, 
	\begin{equation*}
		\eps^{1/4}|\log \eps|^{-1/8}\sum_{x\in\N^*}f(\varepsilon x)\left(\eta(x)-\frac{1}{2}\right)\xRightarrow{\varepsilon\to0}\cY_{\infty}(f)\comma
	\end{equation*}
	for some non-trivial Gaussian field $\cY_{\infty}$ on $\R_{+}$.
\end{conjecture}
\begin{remark}
A finer analysis, keeping track of the coefficients $\mu,D,\sigma$
in \eqref{eq:aEW}, yields  (cf.\ \eqref{eq:var-Y0})
\begin{equation}\label{eq:hyperuniformity_EW}
	\var\!\left(\cY_0^\eps(\psi_L^\eps)\right)\propto\sqrt{\frac{\sigma^{4}}{\mu D}}\ L^{1/2}\log(L)^{1/4}\fstop
\end{equation}
\end{remark}
As a last remark, we note that there is a difficulty interpreting
the advected Edwards-Wilkinson equation \eqref{eq:aEW} on the half-line
as a scaling limit of ${\rm DEP}$: it is not scale invariant, hence not a
direct scaling limit of a discrete model. Moreover, unlike the system
on the bi-infinite line, the boundary does not allow us to change
frame of reference in order to get rid of the advection term and go
back to a scale invariant equation. Nonetheless, if we only look at
the stationary measure, it does seem to have a scale invariant structure.

Let us consider this more closely. Let $\cY^+=\cY^+(t,x)$ be a nontrivial stationary solution of \eqref{eq:aEW} on $\R_+$,
and define a \emph{rescaled field} $\cY_\ell^+=\cY_\ell^+(t,x)$ at the scale $\ell>0$, that is, 
\begin{equation*}
	\cY_\ell^+(t,x)\eqdef\ell^{3/4}\cY^+(\ell^{z}t,\ell x)\comma\qquad t\ge 0\comma x\in \R_+\fstop
\end{equation*}
The exponent $3/4$ is chosen this way for $\cY_{\ell}^+$ to remain
of order $1$, since the fluctuations of the number of particles in the macroscopic interval $[0,\ell]$, corresponding to $\int_0^\ell \cY(t,x)\,\dd x$,
scale (up to logarithmic corrections) as $\ell^{1/4}$. Since we
are interested in the stationary state, the dynamical exponent $z$
remains undetermined. 

By defining a new white noise $\cW_\ell(t,x)=\ell^{(1+z)/2}\,\cW(\ell^{z}t,\ell x)$
with the same law as $\cW$, we obtain for $\cY_\ell^+$ the same equation
\eqref{eq:aEW}, but with rescaled parameters: 
\begin{equation*}
	\mu_{\ell}=\mu\ell^{z-1}\comma\qquad D_{\ell}=D\ell^{z-2}\comma\qquad \sigma_{\ell}=\sigma\ell^{z/2-3/4}\fstop
\end{equation*}
We can see that, indeed, while the equation in \eqref{eq:aEW} is not scale
invariant, the combination $\sqrt{\sigma^{4}/\mu D}$ appearing in
\eqref{eq:hyperuniformity_EW} is, no matter which exponent $z$ we
use to scale time.

\subsection*{Organization of the paper}
The rest of the paper is organized as follows. In Section \ref{sec:model+properties}, 
we introduce the model and prove some of its properties.  
Section \ref{sec:hydro} contains the statements of our results on hydrodynamics. 
In Section \ref{sec:coupling}, we introduce a graphical construction  and a coupling 
 for ${\rm DEP}$, 
thanks to which we prove the key properties of the model
required for the derivation of the hydrodynamic limit, done in Section \ref{sec:proofs}.
 Section \ref{sec:strong-hydro} is devoted to a strong (i.e., in an almost sure sense) version of our hydrodynamic limit. 

\section{Model and first properties}\label{sec:model+properties}

In this section we define our model, derive its attractiveness property, and characterize its set of
extremal (time) invariant  and translation (space) invariant measures.

\subsection{Model}\label{subsec:model}
The \emph{directed exclusion process} (${\rm DEP}$) is a  one-dimensional conservative  interacting particle system with a superposition of two jump mechanisms: a classical  nearest neighbor  symmetric simple  exclusion interaction, plus jumps of particles/holes at distance two, subjected to two constraints, one directional and the other on the value of the occupation variable of the overtaken site. More precisely, ${\rm DEP}$ is 	 the Markov process $(\eta_t)_{t\ge 0}$ with state space $\X:=\{0,1\}^\Z$, and evolving according to the following (pre-)generator, whose action  on local functions $f:\X\to \R$ is given by
\begin{align}\label{eq:generator}
	\cL f(\eta) = \sum_{x,y \in \Z} \eta(x)\left(1-\eta(y)\right) \Gamma_\eta(x,y)\left(f(\eta^{x,y})-f(\eta)\right)\comma\qquad \eta \in \X\fstop
\end{align}
Here, $\eta(x)=1$ (resp.\ $\eta(x)=0$) means that a particle (resp.\ hole) sits at site $x\in \Z$ in the configuration $\eta\in \{0,1\}^\Z$,    $\eta^{x,y}$ denotes the configuration obtained from $\eta$ by exchanging the occupation variables $\eta(x)$ and $\eta(y)$, 	 while
\begin{align*}
	\Gamma_\eta(x,y) :=\begin{dcases}
		1 &\text{if}\  y=x\pm 1\\
		\eta(x+1) &\text{if}\ y=x+2\\
		1-\eta(x-1) &\text{if}\ y=x-2\\
		0 &\text{else}\ .
	\end{dcases}
\end{align*}
This dynamics may be schematically represented  via its four allowed transitions, all occurring at unit rate:
\begin{align}\label{eq:transitions-dep}
	10\to 01\comma\qquad 01\to 10\comma\qquad 110\to 011\comma\qquad 001\to 100\fstop
\end{align} 
In formula \eqref{eq:generator}, transitions were written as particles' jumps from a site $x$ to a site $y$.
Alternatively, if we consider these  transitions as occupation exchanges either for particles or holes (cf.\ \eqref{eq:transitions-dep}), 
we may write $	\cL f(\eta)=\sum_{x\in \Z} \cL_x f(\eta)$,	 with
\begin{align}\label{eq:generator2}
	\cL_xf(\eta)\eqdef\left(f(\eta^{x,x+1})-f(\eta)\right) + \car_{\set{\eta(x)=\eta(x+1)}}\left(f(\eta^{x,x+2})-f(\eta)\right)\comma
\end{align} 
where transitions are all one-sided. 
Along the paper we will use either formula \eqref{eq:generator} or formula \eqref{eq:generator2}. \par\smallskip 

Note that, since  the rates are uniformly bounded, 
that is,
\begin{equation*}
	\sup_{x\in \Z}\sum_{y\in \Z}\sup_{\eta \in \X}\Gamma_\eta(x,y)<\infty\comma
\end{equation*} 
the standard construction in, e.g., \cite[Chapter I]{liggett_interacting_2005} ensures that the operator in \eqref{eq:generator}, defined on local functions, indeed generates  a  Markov-Feller process on $\X$, with corresponding Feller semigroup $(\cP_t)_{t\ge 0}$ on $\cC(\X)$ endowed with the uniform norm. In other words, local functions form a core in $\cC(\X)$ for the corresponding generator.\par\smallskip

\subsection{Attractiveness}\label{subsec:attra}
The very first property that we prove for ${\rm DEP}$ is attractiveness, that is, there exists a  coupling of two copies of ${\rm DEP}$ such that the partial order 
\begin{equation*}
	\xi\le \zeta \qquad\text{if and only if}\qquad \xi(x)\le \zeta(x)\ ,\ x\in \Z\ ,
\end{equation*} is maintained through the (coupled) evolution whenever it holds at the initial time (see, e.g., \cite[Chapter II, Definition 2.3]{liggett_interacting_2005}). The proof goes by checking a recent criterion established in \cite{gobron_saada_couplings_2023}.

\begin{proposition}\label{pr:attractive}
	 ${\rm DEP}$ is attractive.
\end{proposition}
\begin{proof} We verify the two  necessary and sufficient conditions (2.6) and (2.7) in \cite[Theorem 2.4]{gobron_saada_couplings_2023}
	for attractiveness, that we now quote: \par\smallskip

\textit{For any couple of configurations 
	$(\xi,\zeta)\in\X^2$ such that} $\xi\le\zeta$, 
	\begin{enumerate}
		\item[(2.6)] \textit{for all $y\in S$ such that} $\zeta(y)=0$,
		\begin{equation*}
			\sum_{x\in S} \xi(x) \bigl[\Gamma_\xi(x,y) - \Gamma_\zeta(x,y)\bigr]^+ 
			\le
			\sum_{x\in S} \zeta(x) (1 - \xi(x) ) \Gamma_\zeta(x,y)\comma  
		\end{equation*}
		\item[(2.7)] 
		\textit{for all $x\in S$ such that} $\xi(x)=1$,
		\begin{equation*}
			\sum_{y\in S} (1 - \zeta(y) )\bigl [\Gamma_\zeta(x,y) - \Gamma_\xi(x,y)\bigr]^+  
			\le
			\sum_{y\in S} \zeta(y) (1 - \xi(y) ) \Gamma_\xi(x,y)\fstop       
		\end{equation*} 
	\end{enumerate}\smallskip

\noindent
	Thus we fix $\xi, \zeta \in \X$ with $\xi\le \zeta$. 
	As for the first condition, we fix $y\in \Z$ and assume $\zeta(y)=0$. Then, the left-hand side of \cite[Eq.\ (2.6)]{gobron_saada_couplings_2023} reads as
	\begin{align*}
		&	\sum_{x\in \Z}\xi(x)\left[\Gamma_\xi(x,y)-\Gamma_\zeta(x,y)\right]^+ \\
		&=  \xi(y-2)\big[\xi(y-1)-\zeta(y-1)\big]^+
		+ \xi(y+2) \big[(1-\xi(y+1))-(1-\zeta(y+1))	\big]^+
		\\
		&= \xi(y+2)\big[\zeta(y+1)-\xi(y+1)\big]\ ,
	\end{align*}
	which is smaller than or equal to
	\begin{align*}
		&\sum_{x\in \Z}\zeta(x)\left(1-\xi(x)\right)\Gamma_\zeta(x,y)\\
		&= \zeta(y-2)\left(1-\xi(y-2)\right)\zeta(y-1)+ \zeta(y-1)\left(1-\xi(y-1)\right)\\
		& + \zeta(y+1)\left(1-\xi(y+1)\right) + \zeta(y+2)\left(1-\xi(y+2)\right)\left(1-\zeta(y+1)\right)\\
		&\ge \zeta(y+1)\left(1-\xi(y+1)\right)\fstop	\end{align*}
	Hence, the first condition is verified. For what concerns the second one, we fix $x \in \Z$ and assume $\xi(x)=1$. Then, the left-hand side of \cite[Eq.\ (2.7)]{gobron_saada_couplings_2023} reads as
	\begin{align*}
		&\sum_{y\in \Z} \left(1-\zeta(y)\right)\left[\Gamma_\zeta(x,y)-\Gamma_\xi(x,y)\right]^+\\
		& = \left(1-\zeta(x-2)\right)\left[\left(1-\zeta(x-1)\right)-\left(1-\xi(x-1)\right)\right]^+ + \left(1-\zeta(x+2)\right) \left[\zeta(x+1)-\xi(x+1)\right]^+	\\
		&=  \left(1-\zeta(x+2)\right) \left[\zeta(x+1)-\xi(x+1)\right]\comma	
	\end{align*}
	which is smaller than or equal to
	\begin{align*}
		&\sum_{y\in \Z}\zeta(y)\left(1-\xi(y)\right)\Gamma_\xi(x,y)\\
		&= \zeta(x-2)\left(1-\xi(x-2)\right)\left(1-\xi(x-1)\right) + \zeta(x-1)\left(1-\xi(x-1)\right)\\
		&+ \zeta(x+1)\left(1-\xi(x+1)\right) + \zeta(x+2)\left(1-\xi(x+2)\right)\xi(x+1)\\
		&\ge \zeta(x+1)\left(1-\xi(x+1)\right)\fstop
	\end{align*}
	This verifies the second condition in \cite[Theorem 2.4]{gobron_saada_couplings_2023}, thus concluding the proof of the proposition.
\end{proof}
\begin{remark}
	The two main inequalities in the proof above are not necessarily strict. Taking, for instance,	 
	$\zeta(y-2)=\zeta(y-1)=\xi(y+1)=0$ and $\zeta(y+1)=\xi(y+2)=1$ implies that the first inequality is an equality. This will prevent us to use \cite[Theorem~2.9, Item~2]{gobron_saada_couplings_2023} to prove Proposition \ref{prop:IcapSe} below.
\end{remark}

\subsection{Invariant  and translation invariant measures}\label{subsec:inv-meas}
Let $\cI$ denote the subset of probability measures on $\X$ which are invariant (stationary) for ${\rm DEP}$. We start by checking that the Bernoulli product measures $(\nu_\rho)_{\rho \in [0,1]}$, with $\nu_\rho(\eta(0))=\rho$, 
are invariant for ${\rm DEP}$.
\begin{proposition}\label{pr:invariance-product}
	For all $\rho \in [0,1]$, we have $\nu_\rho \in \cI$.	
\end{proposition}
\begin{proof}
	Since local functions are a core for the generator $\cL$, by linearity, it suffices to check $	\nu_\rho(\cL f) =0$, 
	for every finite subset $A\subset \Z$ and function 	$f:\X\to \R$ of the form $f(\eta)=\prod_{x\in A}\eta(x)$. Furthermore, since $\nu_\rho$ is product and $A\subset \Z$ can be taken to be finite, the invariance of $\nu_\rho$ follows from the invariance of $\nu_\rho^n\eqdef \otimes_{i\in \T_n}{\rm Bern}(\rho)$ with respect to  ${\rm DEP}$  on the torus $\T_n\eqdef (\Z/n\Z)$, evolving on $\X_n:= \{0,1\}^{\T_n}$ and with generator  $\cL_n:=\sum_{x\in\T_n}\cL_x$ (with $\cL_x$ defined as in \eqref{eq:generator2}),  for all $n\in \N^*$ large enough. 
	
	Let us fix $n\in \N^*$, and show that 
	\begin{align*}
		\sum_{\eta'\in \X_n}\nu_\rho^n(\eta')\, \cL_n\car_\eta(\eta')=0\comma\qquad \eta \in \X\comma
	\end{align*}
	which, since $\nu_\rho^n$ is constant, is equivalent to
	\begin{align}\label{eq:identity-assaf}
		\sum_{\substack{\eta'\in \X_n\\
				\eta'\neq \eta}} \cL_n(\eta,\eta')=	\sum_{\substack{\eta'\in \X_n\\
				\eta'\neq \eta}} \cL_n(\eta',\eta)\comma\qquad \eta \in \X_n\comma
	\end{align}
	where $\cL_n(\eta,\eta')\ge 0$ denotes the jump rate from $\eta$ to $\eta'\in \X_n$. If we consider the SEP-part of the jumps, 
	the above identity clearly holds true. For the remaining part of the jump rates, we note that the left-hand side above is equal 
	to the number of blocks of occupied sites of size at least two in $\eta$ (corresponding to jumps $11...1110\to 11...1011$) + the number 
	of blocks of empty sites of size at least two in $\eta$ (corresponding to jumps $00...0001\to 00...0100$). Analogously, 
	the right-hand side above is equal to the number of blocks of empty sites of size at least  two in $\eta$ (corresponding to jumps 
	$0010...00\to 1000...00$) + the number of blocks of occupied sites of size at least two in $\eta$ (corresponding to jumps 
	$1101...11\to 0111...11$). This proves identity \eqref{eq:identity-assaf}, thus yielding the desired result.
\end{proof}

Let $\tau_x$, $x\in \Z$, denote the space shift by $x$, which acts on configurations $\eta\in \{0,1\}^\Z$ as $\tau_x \eta= \eta(\emparg+x)$, and on measures $\mu$ on $\{0,1\}^\Z$ as $\tau_x \mu = \mu\circ \tau_x^{-1}$.
Let $\cS$ denote the subset of probability measures on $\X$ which are translation invariant, i.e., $\mu \in \cS$ if and only if $\tau_x \mu=\mu$ for all $x\in \Z$. Further,  $\cS_e$ and $(\cI\cap \cS)_e$ stand respectively for the extremal subsets of $\cS$ and $\cI\cap \cS$. 
This is our main result of this part.
\begin{proposition}\label{prop:IcapSe}
	$(\cI\cap \cS)_e= (\nu_\rho)_{\rho \in [0,1]}$.
\end{proposition}

We will prove this proposition in Section \ref{sec:coupling} below as it requires the use of the coupling introduced there.

\section{Hydrodynamic limit and propagation of local equilibrium}\label{sec:hydro} 
We show that, for suitably initialized particle systems 
and under the hyperbolic space-time scaling,  ${\rm DEP}$ converges 
(in the  sense of propagation of local equilibrium) to the  scalar conservation law \eqref{eq:HDL} on $\R$. As most common in  translation invariant settings, the macroscopic flux of $\rm DEP$ particles through the origin is described by the function $G_\DEP$ therein, while the hydrodynamic density profile is described by the corresponding entropy solutions $u(\emparg,\emparg)$. 
Before presenting our main results, let us examine  \eqref{eq:HDL} more closely, by checking that $G_\DEP$ is indeed the correct macroscopic flux arising from ${\rm DEP}$ (for the discussion on well-posedness of the Cauchy problem and entropy solutions, see Section \ref{sec:entropy-solutions}). 

Recall \eqref{eq:generator} and \eqref{eq:generator2}, and compute, for all $\eta\in\X$ and  $x\in \Z$,
\begin{equation}\label{eq:Leta}
	\begin{aligned}
		\cL \eta(x)&= \cL_{x-2}\eta(x)+\cL_{x-1}\eta(x)+\cL_x\eta(x)\\
		&= \car_{\set{\eta(x-2)=\eta(x-1)}}\tonde{\eta(x-2)-\eta(x)}
		\\
		&\qquad+ \eta(x-1)-\eta(x)+\eta(x+1)-\eta(x)\\
		&\qquad + \car_{\set{\eta(x)=\eta(x+1)}}\tonde{\eta(x+2)-\eta(x)}
		\fstop
	\end{aligned}
\end{equation}
Hence,  the microscopic flux across site $0$ is 
\begin{align}\label{eq:jeta}
	\begin{aligned}
		j(\eta)&=j_0(\eta):= \cL\quadre{\sum_{x>0}\eta(x)} \\
		&=\quadre{\eta(0)-\eta(1)}
		+ \quadre{\eta(-1)\eta(0)\tonde{1-\eta(1)}-\tonde{1-\eta(-1)}\tonde{1-\eta(0)}\eta(1)}\\
		&\qquad+ \quadre{\eta(0)\eta(1)\tonde{1-\eta(2)}-\tonde{1-\eta(0)}\tonde{1-\eta(1)}\eta(2)}\fstop
	\end{aligned}
\end{align}
The above definition \eqref{eq:jeta} of $j(\eta)$ is partly formal, as the function $\sum_{x>0}\eta(x)$ does
not belong to the domain of the generator $\cL$. Nevertheless, the formal
computation gives rise to a well-defined function \eqref{eq:jeta}, as the ${\rm DEP}$'s rates are
local functions. 
Finally,  taking expectation with respect to any element in $(\cI\cap \cS)_e=(\nu_\rho)_{\rho\in[0,1]}$ 
(Proposition \ref{prop:IcapSe}) yields
\begin{equation*}
	\nu_\rho(j)= 2\tttonde{\rho^2\tonde{1-\rho}-\tonde{1-\rho}^2\rho} = G_\DEP(\rho)\comma\qquad \rho \in [0,1]\comma
\end{equation*}
that is, the macroscopic flux is indeed the expectation of the microscopic flux.

Our first main result is ${\rm DEP}$'s hydrodynamic limit. In what follows, 
$\eps\in (0,1)$ satisfies $\eps^{-1}\in \N^*$, and $\cC_c(\R)$ is the space of continuous, 
compactly supported functions on $\R$.

\begin{theorem}[Hydrodynamic limit]\label{th:HDL}
	Let  $u_0:\R\to [0,1]$  be a measurable function, and let $(\mu_\eps)_\eps$ be a sequence of probability measures on $\X$ associated to the profile $u_0$, i.e.,
	\begin{equation*}
		\mu_\eps\bigg(\bigg|\,	\eps\sum_{x\in \Z} f(\eps x)\,\eta(x)-\int_{\R} f(x)\,u_0(x)\,\dd 	x\,\bigg|>\delta\bigg)\xrightarrow{\eps\to 0}0\comma\quad \delta >0\comma f \in \cC_c(\R)\fstop
	\end{equation*}
	Then,  letting $u(\emparg,\emparg):\R\times [0,\infty)\to [0,1]$ 
	denote the
	entropy solution  of \eqref{eq:HDL} (cf.\ Section \ref{sec:entropy-solutions}) with initial condition $u_0$, we have, for all  $T>0$, 
	\begin{align*}
		\P_{\mu_\eps}\!\bigg(\sup_{t\in [0,T]}\bigg|\eps\sum_{x\in \Z}f(\eps x)\,\eta_{t\eps^{-1}}(x) - \int_{\R} 
		f(x)\,u(x,t)\,\dd x\,\bigg|>\delta\bigg)\xrightarrow{\eps\to 0}0\comma\quad \delta >0\comma f \in \cC_c(\R)\comma	
		\end{align*}
	where $\P_{\mu_\eps}$ denotes the law of ${\rm DEP}$ when initialized according to $\mu_\eps$.
\end{theorem}

\begin{remark}[Conservation of local equilibrium]
	A natural strengthening of the hydrodynamic limit in
	Theorem~\ref{th:HDL} is the conservation of local equilibrium.
	For general Cauchy initial data, a strong, pointwise formulation of this
	property for ${\rm DEP}$ would read as follows. Let
	$u_0:\R\to[0,1]$ be measurable and let $(\mu_\eps)_\eps$ be product
	measures associated with $u_0$, that is,
	\begin{equation*}
		\mu_\eps(\eta(x)\in\emparg)
		=
		\nu_{u^{\eps,x}}(\eta(x)\in\emparg)
		\comma\qquad x\in\Z\comma
	\end{equation*}
	for some $(u^{\eps,x})_{\eps,x}\subset[0,1]$ satisfying
	\begin{equation*}
		\int_K
		\tabs{u^{\eps,\lfloor x\eps^{-1}\rfloor}-u_0(x)}
		\,\dd x
		\xrightarrow{\eps\to0}0
		\comma\qquad K\subset\R\ \text{compact}\fstop
	\end{equation*}
	Then, writing $u$ for the entropy solution of \eqref{eq:HDL} with
	initial condition $u_0$, one would expect that, for every $t>0$ and
	every continuity point $x\in\R$ of $u(\emparg,t)$,
	\begin{equation}\label{eq:expected-local-eq}
		\lim_{\eps\to0}
		\tau_{\lfloor x\eps^{-1}\rfloor}
		\tttonde{\mu_\eps\cP_{t\eps^{-1}}}
		=
		\nu_{u(x,t)}\fstop
	\end{equation}
	
	The pointwise form \eqref{eq:expected-local-eq} is the analogue for
	${\rm DEP}$ of the conservation of local equilibrium established in
	\cite{landim_conservation_1993}; see also
	\cite[Chapter~IX, Theorem~0.2]{kipnis_scaling_1999}.
	A weaker form, in which local equilibrium is obtained after macroscopic
	spatial averaging, was proved for attractive misanthrope processes in
	\cite[Theorem~7.1]{rezakhanlou_hydrodynamic_1991}; see also
	\cite[Chapter~VIII, Section~4]{kipnis_scaling_1999}.
	The strong Euler-scale result of
	\cite[Theorem~3]{landim_conservation_1993} assumes that the macroscopic
	flux is strictly convex or strictly concave, an assumption also entering
	the corresponding result in \cite[Chapter~IX]{kipnis_scaling_1999}.
	The flux
	$G_\DEP(\rho)=2\rho(1-\rho)(2\rho-1)$ is instead neither convex nor
	concave on $[0,1]$, and ${\rm DEP}$ does not fall directly within the
	class of misanthrope processes treated in
	\cite{rezakhanlou_hydrodynamic_1991}. Thus, for general Cauchy initial
	data, proving conservation of local equilibrium for ${\rm DEP}$ remains
	open and appears to require a substantial extension of the available
	methods, which we plan to pursue in future work. For Riemann (i.e.,
	step) initial data, conservation of local equilibrium is instead proved
	below in Theorem~\ref{th:local-eq-riemann}.
\end{remark}

 The constructive method we use consists in proving Theorem \ref{th:HDL} first in the Riemann case, that is, 
when the initial density profile is a one-step function: 
\begin{equation*}
	u_0=\lambda\car_{(-\infty,0)}+\rho \car_{[0,+\infty)}\comma\qquad \text{for some}\ \lambda, \rho \in [0,1]\fstop
	 \end{equation*}
The entropy solution $u(x,t)$ of \eqref{eq:HDL} is then given by a variational formula 
and can be explicitly computed, as in \cite{bahadoran_constructive_2002} (there, this computation is 
explicit for various examples). 
We compute it in Section \ref{sec:riemann}
after a reminder on entropy solutions in Section \ref{sec:entropy-solutions}. We then derive in
Theorem \ref{th:local-eq-riemann} the conservation of local equilibrium in the Riemann case.
We finally  outline the proof of  Theorem \ref{th:HDL}, that is, the hydrodynamic results
for Cauchy initial data, in Section \ref{subsec:cauchy}. This proof
relies on an approximation scheme similar to Glimm's scheme  for hyperbolic systems of
conservation laws (see, e.g., \cite{serre_systems_1999}), under which it  
is possible to reconstruct any entropy solution by approximations using only Riemann
solutions. 
 It requires two crucial properties of the model, 
\textit{macroscopic stability} and \textit{finite propagation}. It also requires a  monotonicity property,
that is, the preservation of stochastic order of an arbitrary number of copies of the model.

 These three properties
deal with coupling, the subject of the next section. There, we prove the existence of the model via a graphical representation,
which complements the analytical description given in Section \ref{sec:model+properties}. This graphical representation
enables to define couplings that not only preserve monotonicity, 
thus granting attractiveness of the model, but also general monotonicity. 
Thanks to the properties of this coupling, in Section \ref{sec:coupling} 
we prove 
Proposition  \ref{prop:IcapSe}, as well as macroscopic stability and 
finite propagation for  ${\rm DEP}$.

\section{A new coupling: definition and properties}\label{sec:coupling}

In this section, we construct the main coupling that we use throughout the paper, 
and prove some of its properties.
Before defining the coupling, we will describe a graphical construction
of $\rm DEP$ (see, e.g., \cite{harris_1972,harris_1978,durrett_saint-flour,liggett_interacting_2005,liggett_1999}).

\subsection{Graphical construction and coupling}\label{sec:graphical}
 For all $x\in\Z$,
 define the transformation $\Phi_{x}:\mathbf{X}\to\mathbf{X}$ (analogous to the mapping $\cT$ in \cite[Section 6]{bgrs5}) as
\begin{align}\label{def:phix}
		\Phi_x(\eta)\eqdef \begin{dcases}
			\eta^{x,x+1} &\text{if}\ \eta(x)\neq \eta(x+1)\\
			\eta^{x,x+2} &\text{if}\ \eta(x)=\eta(x+1)\comma
		\end{dcases}
\end{align}
where we recall that $\eta^{x,y}$ denotes exchange of occupation numbers.
Then (cf.\  \eqref{eq:generator2}),
\begin{equation}\label{eq:generator_with_phi}
\cL_{x}f(\eta)=f(\Phi_{x}(\eta))-f(\eta)\comma\qquad x\in \Z\comma\eta \in \X\fstop
\end{equation}
Hence, $\rm DEP$ consists in applying with rate $1$, for every $x\in\Z$,
the transformation $\Phi_x$. 
Equivalently, one may apply $\Phi_{x}$
with rate $1$ if $\eta(x)=1$, and apply $\Phi_{x}$ with rate $1$ if $\eta(x)=0$,  as clarified below.  
Although this last formulation seems like an over-complication
of the first one, it will turn out to be useful when defining the
coupling.

Let us make this discussion more detailed. Consider two independent rate $1$ Poisson processes 
$\omega_\alpha$, $\alpha \in \{0,1\}$.
More precisely, letting $(\Xi,\cG)$ be the measurable space of $\sigma$-finite $\N$-valued measures  ($\N\eqdef \{0,1,\ldots\}$)
on $\Z\times \R_+$ (endowed with the $\sigma$-field $\cG$ induced by the mappings $\Xi\ni m\mapsto m(A)\in \N$, 
with $A\subset  \Z\times \R_+$ being any Borel set), $\P$ denotes the  unique law on the product space $(\Omega,\cF)=(\Xi^2,\cG^{\otimes 2})$ 
for which a random element $\omega=(\omega_0,\omega_1)$ is distributed as two independent Poisson point processes on $\Z\times \R_+$ with unit intensity. 
We have, for $x\in\Z$, $\eta\in\X$ and $t\in \R_+$, 
when  $\omega_\alpha(\{x\}\times \{t\})=1$, $\alpha \in \{0,1\}$, corresponding to a clock ring of the Poisson clock $\omega_\alpha$ at site $x$ and time $t$, then
\begin{equation}\label{eq:eta-update-alpha}
	\eta_t=\begin{dcases} \Phi_x(\eta_{t^-}) &\text{if}\ \eta_{t^-}(x)=\alpha\\
		\eta_{t^-} &\text{if}\ \eta_{t^-}(x)\neq \alpha\fstop
		\end{dcases}
\end{equation}
 In other words, if, say, $\omega_1(\{x\}\times\{t\})=1$, then
nothing happens when $\eta_{t^-}(x)=0$, whereas
$\eta_t=\Phi_x(\eta_{t^-})$ when $\eta_{t^-}(x)=1$. This provides an alternative graphical construction whose occupation process,
for every initial configuration, has the same law as the process obtained by
applying each $\Phi_x$ at rate $1$, as described in
\eqref{def:phix}--\eqref{eq:generator_with_phi}.

Write $\overline{\omega}\eqdef \omega_0+\omega_1$ (we have that $\P$-a.s.\ and for all $t\in \R_+$, $\overline{\omega}(\Z\times \{t\})\in \{0,1\}$), 
and let $\E$ denote the corresponding expectation.

Then, fixing an initial configuration $\eta_{0}\in \X$, for $\P$-a.e.\ $\omega \in \Omega$, there exists
 a unique mapping
\begin{equation}\label{def:graph-eta_t}
	t\in \R_+\cup\{0\}\longmapsto \eta_t=\eta(\eta_0,\omega,t)\in\X
\end{equation}
satisfying: 
\begin{enumerate}[(a)] 
	\item\label{it:a} $t\mapsto \eta(\eta_0,\omega,t)$ is right-continuous ($\X$ is endowed with the product discrete topology);
		\item\label{it:b}  $\eta(\eta_0,\omega,0)=\eta_0$;
		\item\label{it:c} for all $t\in\R_+$ and $ x\in\Z$, $\eta(\eta_0,\omega,t)=
			\Phi_x(\eta(\eta_0,\omega,t^-))$ if
		\[
			 \omega_\alpha(\{x\}\times \{t\})=1\quad \text{and}\quad \eta(\eta_0,\omega,t^-)(x)=\alpha\comma \quad 
			 \text{for some}\  \alpha \in \{0,1\}\comma
	\]
	while $\eta(\eta_0,\omega,t)=
	\eta(\eta_0,\omega,t^-)$ otherwise;
		\item\label{it:d}  for all $0\le s<t$ and $x\in \Z$,
		\[\overline{\omega}([s,t]\times \{x-2,x-1,x\})=0\quad\Longrightarrow\quad 
		\eta(\eta_0,\omega,r)(x)=\eta(\eta_0,\omega,s)(x)\comma r \in [s,t]\fstop\]
	\end{enumerate}

\par\noindent

The process obtained using this mapping is, indeed, ${\rm DEP}$. Moreover, remark that condition \ref{it:c} 
states that $ \omega_\alpha(\{x\}\times\{t\})=1$, $\alpha \in \{0,1\}$, is an update time at $x\in \Z$ if and only if 
$\eta(\eta_0,\omega,t^-)(x)=\alpha$; while condition \ref{it:d} states that the system cannot be modified otherwise. \par\smallskip

We can now use this construction in order to define our coupling. First, let us introduce a probability space $(\Omega_{0},\mathcal{F}_{0},\P_{0})$
of initial conditions. When coupling two copies of the process, we
may pick $\Omega_{0}=\X^{2}$ for a pair of initial configurations
$(\eta_{0}^{1},\eta_{0}^{2})$; in general, 
we may take a (possibly countably infinite) sequence of initial configurations
$(\eta_{0}^{1},\eta_{0}^{2},\dots)$, considered as a random variable 
 on the product space $\Omega_{0} = \X^\cJ$, for some index set $\cJ$. 
 Let $\widetilde \P = \P_0\otimes \P$ be a product  law on 
 $\widetilde \Omega:=\Omega_{0}\times\Omega$, whose marginal on $\Omega$ 
 coincides with  $\P$ given above. Then, the total process $(\eta^j)_{j\in \cJ}$  
 is  constructed on the space $\widetilde{\Omega}$ by setting, 
for $\widetilde \P$-a.e.\ $\widetilde \omega=(\omega^0,\omega)$ 
(cf.\ \eqref{def:graph-eta_t}), 
\begin{equation}\label{eq:etatj}
t\in \R_+\cup \{0\}\longmapsto	\eta_t^j= \eta(\eta_0^j(\omega^0),\omega,t)\in \X\comma\qquad j\in \cJ\fstop
\end{equation}
Note that each marginal $\eta^j$ is initialized according to $\eta_0^j$, but all use a common underlying Poisson processes 
$\omega=(\omega_0,\omega_1)$. As can be seen from the  above 
definition, each $\eta^j$ evolves like ${\rm DEP}$, with initial condition $\eta_0^j$, 
where $\eta_0=(\eta_0^1,\eta_0^2,\ldots)\sim \P_0$. When the initial state is nonrandom, 
we shall simply write $\eta_t^j=\eta(\eta_0^j,\omega,t)$.\par\smallskip

In the rest of the section, we focus on the coupling of two copies
of $\rm DEP$, referred to as $\zeta$ and $\xi$ (rather that $\eta^{1}$
and $\eta^{2}$). When coupling more than two copies using this construction,
these results will hold pairwise, simultaneously for all pairs.\par\smallskip

 To illustrate the coupling, fix $x\in\Z$ and $t>0$ such that, say,
$\omega_1(\{x\}\times\{t\})=1$. By \eqref{eq:eta-update-alpha} and
\eqref{eq:etatj}, there are four cases: (a) if
$\zeta_{t^-}(x)=\xi_{t^-}(x)=1$, both configurations are updated, namely,
$\zeta_t=\Phi_x(\zeta_{t^-})$ and $\xi_t=\Phi_x(\xi_{t^-})$; (b) if
$\zeta_{t^-}(x)=\xi_{t^-}(x)=0$, neither configuration changes; (c) if
$\zeta_{t^-}(x)=0$ and $\xi_{t^-}(x)=1$, only $\xi$ is updated, with
$\xi_t=\Phi_x(\xi_{t^-})$; and (d) if $\zeta_{t^-}(x)=1$ and
$\xi_{t^-}(x)=0$, only $\zeta$ is updated, with
$\zeta_t=\Phi_x(\zeta_{t^-})$.

\begin{definition}\label{def:disc}
	In a  coupled process  $(\zeta_t,\xi_t)_{t\ge 0}$, there is a \emph{discrepancy} 
	at site $z\in \Z$ at time $t\geq 0$ if $\zeta_t(z)\ne\xi_t(z)$. Further, we say that the discrepancy is  
	\emph{positive} if $\zeta_t(z)>\xi_t(z)$, and \emph{negative} if $\zeta_t(z)<\xi_t(z)$.
\end{definition}

The dynamics of discrepancies is described by the following proposition.

\begin{proposition}\label{prop:properties-disc}
Under the coupling above:
\begin{enumerate}[(a)]
\item the number of discrepancies cannot increase;
\item discrepancies move on the line, keeping the same sign and never swapping
positions;
\item if there is a discrepancy at $x$ but none at $x\pm1$, the discrepancy
will move to $x\pm1$ with rate at least $1$.
\item neighboring discrepancies with opposite sign annihilate each other with rate
at least $2$. 
\item  A discrepancy cannot jump over another discrepancy of opposite sign. 
\end{enumerate}
\end{proposition}

\begin{proof}
Let  us abbreviate  $\omega_\alpha(x)=(\omega_\alpha(\{x\}\times(0,t]))_{t\in \R_+}$, $x\in \Z$ and $\alpha\in \{0,1\}$.
Let us verify the first two properties at every clock ring (recall \eqref{eq:eta-update-alpha}): Without loss
of generality (by particle-hole symmetry), we may assume that the
clock that rang is $\omega_{1}(x)$,  say at some time $s$.  If  $\zeta_{s^-}(x)=\xi_{s^-}(x)=0$  nothing
happens, otherwise we consider two cases:
\begin{enumerate}
\item If  $\zeta_{s^-}(x)=\xi_{s^-}(x)=1$,  then we apply $\Phi_{x}$ to both configurations.
Below we represent all nontrivial transitions (the first line represents $\zeta$ and the second
$\xi$, while the first, second and third columns correspond to sites $x$, $x+1$ and $x+2$, respectively):
\begin{equation*}
\begin{tabular}{ccc||ccc}
\hline 
\begin{tabular}{ccc}
1 & 0 & 0\tabularnewline
1 & 0 & 1\tabularnewline
\end{tabular} & $\to$ & %
\begin{tabular}{ccc}
0 & 1 & 0\tabularnewline
0 & 1 & 1\tabularnewline
\end{tabular} & %
\begin{tabular}{ccc}
1 & 0 & 1\tabularnewline
1 & 1 & 0\tabularnewline
\end{tabular} & $\to$ & %
\begin{tabular}{ccc}
0 & 1 & 1\tabularnewline
0 & 1 & 1\tabularnewline
\end{tabular}\tabularnewline
\hline 
\hline 
\begin{tabular}{ccc}
1 & 0 & 0\tabularnewline
1 & 1 & 0\tabularnewline
\end{tabular} & $\to$ & %
\begin{tabular}{ccc}
0 & 1 & 0\tabularnewline
0 & 1 & 1\tabularnewline
\end{tabular} & %
\begin{tabular}{ccc}
1 & 0 & 1\tabularnewline
1 & 1 & 1\tabularnewline
\end{tabular} & $\to$ & %
\begin{tabular}{ccc}
0 & 1 & 1\tabularnewline
1 & 1 & 1\tabularnewline
\end{tabular}\tabularnewline
\hline 
\hline 
\begin{tabular}{ccc}
1 & 0 & 0\tabularnewline
1 & 1 & 1\tabularnewline
\end{tabular} & $\to$ & %
\begin{tabular}{ccc}
0 & 1 & 0\tabularnewline
1 & 1 & 1\tabularnewline
\end{tabular} & %
\begin{tabular}{ccc}
1 & 1 & 0\tabularnewline
1 & 1 & 1\tabularnewline
\end{tabular} & $\to$ & %
\begin{tabular}{ccc}
0 & 1 & 1\tabularnewline
1 & 1 & 1\tabularnewline
\end{tabular}\tabularnewline
\hline 
\end{tabular}
\end{equation*}

\item Without loss of generality  $\zeta_{s^-}(x)=1$ and $\xi_{s^-}(x)=0$,  so we apply
$\Phi_{x}$ to the first line leaving the second fixed. 
As similarly done in the table above, all transitions read as follows 
($\diamond$ and $\ast$ represent either $0$ or $1$):
\begin{equation*}
\begin{tabular}{ccc||ccc}
\hline 
\begin{tabular}{ccc}
1 & 0 & ${\diamond}$\tabularnewline
0 & 0 & ${\ast}$\tabularnewline
\end{tabular} & $\to$ & %
\begin{tabular}{ccc}
0 & 1 & ${\diamond}$\tabularnewline
0 & 0 & ${\ast}$\tabularnewline
\end{tabular} & %
\begin{tabular}{ccc}
1 & 0 & ${\diamond}$\tabularnewline
0 & 1 & ${\ast}$\tabularnewline
\end{tabular} & $\to$ & %
\begin{tabular}{ccc}
0 & 1 & ${\diamond}$\tabularnewline
0 & 1 & ${\ast}$\tabularnewline
\end{tabular}\tabularnewline
\hline 
\hline 
\begin{tabular}{ccc}
1 & 1 & 0\tabularnewline
0 & ${\ast}$ & 0\tabularnewline
\end{tabular} & $\to$ & %
\begin{tabular}{ccc}
0 & 1 & 1\tabularnewline
0 & ${\ast}$ & 0\tabularnewline
\end{tabular} & %
\begin{tabular}{ccc}
1 & 1 & 0\tabularnewline
0 & ${\ast}$ & 1\tabularnewline
\end{tabular} & $\to$ & 
\begin{tabular}{ccc}
0 & 1 & 1\tabularnewline
0 & ${\ast}$ & 1\tabularnewline
\end{tabular}\tabularnewline
\hline 
\end{tabular}
\end{equation*}
\end{enumerate}
A close inspection of these transitions shows  the first and second properties.

For the third property, first assume there is a discrepancy at $x$ and none
at $x+1$. Without loss of generality, we can consider the following
cases  (we do not write time $s$ for readability): 
\begin{enumerate}
\item $\zeta(x)=1$, $\xi(x)=0$, $\zeta(x+1)=\xi(x+1)=1$. Then a ring
of $\omega_{0}(x)$ will move the discrepancy to the right.
\item $\zeta(x)=1$, $\xi(x)=0$, $\zeta(x+1)=\xi(x+1)=0$. Then a ring
of $\omega_{1}(x)$ will move the discrepancy to the right.
\end{enumerate}
Similarly when there is a discrepancy at $x$ but none at $x-1$ one
of the clocks $\omega_{0}(x-1)$ or $\omega_{1}(x-1)$ will move it
to the left.

 For the fourth property,  two neighboring discrepancies of opposite signs at $x$ and
$x+1$ annihilate each other when either $\omega_{0}(x)$ or $\omega_{1}(x)$
rings. For example, if the discrepancy at $x$ is positive and at $x+1$  is
negative, then $\zeta(x)=1$, $\xi(x)=0$, $\zeta(x+1)=0$, $\xi(x+1)=1$. A ring of 
$\omega_1(x)$ will thus cause the particle at $x$ to jump to $x+1$ for $\zeta$, 
leaving $\xi$ unchanged.
A ring of $\omega_0(x)$, on the other hand, will  cause the $\xi$-particle at $x+1$ 
to jump to $x$, leaving $\zeta$ unchanged. In both cases the discrepancies annihilate
each other.

 Finally, for the fifth property, consider an update at $x$. Since each
discrepancy moves by at most two sites, it could jump over one of opposite sign
only if the two discrepancies were located at $x,x+1$ or at $x+1,x+2$. In the
first case, only one configuration is updated, say $\zeta$, and the opposite
signs imply that $\zeta(x)\neq\zeta(x+1)$. Hence, $\Phi_x$ performs a
nearest-neighbor exchange, and the two discrepancies annihilate. In the second
case, if there is no discrepancy at $x$, either neither or both configurations
are updated; in the latter case, one performs a nearest-neighbor exchange and
the other a range-two exchange, after which they coincide on
$\{x,x+1,x+2\}$. If there is also a discrepancy at $x$, the case in which it
has sign opposite to that at $x+1$ has already been considered. Thus, it
remains only to consider the case in which the discrepancies at $x$ and
$x+1$ have the same sign; the range-two exchange then annihilates those at
$x$ and $x+2$. In all cases, discrepancies of opposite sign cannot jump over
each other. 
\end{proof}

\subsection{Consequences of the coupling}\label{subsec:cons-coupl}
We now collect some consequences of the coupling and its properties 
(Proposition \ref{prop:properties-disc}). Since we consider nonrandom initial configurations, 
all statements hold $\P$-a.s.\ (rather than $\widetilde \P$-a.s.).
\begin{corollary}[Attractiveness]\label{cor:monotone} 
The coupling is monotone. 
	In particular, this gives an alternative proof of attractiveness 
	(Proposition \ref{pr:attractive}).
\end{corollary}

\begin{proof}
	Saying that $\xi\le\zeta$ is the same as saying that all discrepancies
	are positive. Since discrepancies cannot be created or change sign,
	$\xi_{0}\le\zeta_{0}$ implies, $\P$-a.s., $\xi_{t}\le\zeta_{t}$, for all $t>0$.
\end{proof}

The following property 
will be crucial in proving the hydrodynamic limit.

\begin{corollary}[Exact macroscopic stability]\label{cor:macrostab}
	$\P$-a.s.,  for all $t\ge 0$ and finite initial configurations $\zeta,\xi\in \X$, 
	\begin{equation*}
		\varDelta(\zeta_t,\xi_t)\le \varDelta(\zeta,\xi)\comma\qquad 
		\text{with}\ \varDelta(\zeta,\xi)\eqdef \sup_{x\in \Z}\bigg|\sum_{y\le x}\tonde{\zeta(y)-\xi(y)}\bigg|\fstop \end{equation*} 
\end{corollary}

\begin{proof}
We observe that $\sum_{y\le x}\left(\zeta(y)-\xi(y)\right)$ can be
seen as the sum of (signed) discrepancies up to position $x\in \Z$. As we proved in  
Proposition \ref{prop:properties-disc} that discrepancies are never created, 
and that opposite-sign discrepancies cannot swap positions 
nor jump over each other,   
$\varDelta(\zeta_t,\xi_t)$ cannot increase in time, proving macroscopic stability.
\end{proof}

As an immediate consequence of the coupling's properties and the fact that ${\rm DEP}$ 
only allows for finite-range jumps, information propagates at finite speed. 
Since ${\rm DEP}$ has bounded rates and interaction range $2$, disturbances cannot propagate
arbitrarily fast. More precisely, one has the following standard finite-propagation estimate.

\begin{proposition}[Finite propagation]\label{prop:finite-prop}
	There exist constants $v, C>0$ such that the following holds. For any $x<y$ in $\Z$,
	any $(\zeta_0,\xi_0)\in \X^2$, and any
	\begin{equation*}
	0<t<\frac{y-x}{2v}\comma
	\end{equation*}
	if $\eta_0$ and $\xi_0$ coincide on the interval $[x,y]\cap\Z$, then
	\begin{equation*}
	\P\Bigl(
	\eta(\zeta_0,\omega,s)(z)=\eta(\xi_0,\omega,s)(z)
	\ \text{for all } z\in [x+vt,y-vt]\cap\mathbb Z
	\text{ and } s\in[0,t]
	\Bigr)
	\ge 1-e^{-Ct}\fstop
	\end{equation*}
\end{proposition}

\begin{proof}
	This is the standard finite-propagation estimate for one-dimensional particle
	systems with bounded rates and finite-range jumps/interactions; see, for instance, 
	\cite[Lemma 5.2]{bahadoran_euler_2006}, \cite[Remark 4.1]{bgrs3}, and references therein. Since in ${\rm DEP}$ every update only involves sites at distance at
	most $2$, the same argument applies here. 
\end{proof}

\begin{remark}\label{rk:toom-no-finite-prop} In contrast to ${\rm DEP}$, the Toom 
	model \cite{crawford2016invariance,crawford_kozma_toom_2020} lacks this
	finite propagation property.
	\end{remark}

To conclude this section, we go back to Proposition \ref{prop:IcapSe}.

\begin{proof}[Proof of Proposition \ref{prop:IcapSe}] 
	As noticed at the end of the proof of Proposition  \ref{pr:attractive}, 
	we cannot apply \cite[Theorem~2.9, Item~2]{gobron_saada_couplings_2023}, since it relied on 
	\cite[Proposition 3.11]{gobron_saada_couplings_2023}. That proposition  required only 
	sufficient assumptions on the attractiveness inequalities to be combined with a coupling 
	introduced in that paper. Therefore, we rely on the coupling we introduced in this section. 
	Using this in combination with  Liggett's strategy 
	(see, e.g., \cite[Chapter VIII.2]{liggett_interacting_2005}) and 
	Proposition \ref{pr:invariance-product} yields that the Bernoulli 
	product measures are the only extremal elements of $\cI\cap \cS$.

Let us recall the main steps of this proof. For any pair $(\pi,\mu)$
of translation invariant stationary measures of ${\rm DEP}$, we can construct
a translation invariant stationary coupling $(\zeta,\xi)$. Under this coupling, the probability of neighboring discrepancies
with opposite sign is zero: for all $x\in \Z$,
\begin{equation*}
\P(\zeta(x)>\xi(x),\,\zeta(x+1)<\xi(x+1))=\P(\zeta(x)<\xi(x),\,\zeta(x+1)>\xi(x+1))=0\fstop
\end{equation*}
Indeed, for any $N>0$, let $E_N$ be the cardinality of the set of discrepancies in $[1,N+1]$, and $D_{N}\le E_N/2$ be the cardinality  of the set of neighboring positive-negative discrepancy pairs  in $[1,N+1]$, i.e.,
\begin{equation*}
	\left\{ x\in[1,N]\cap \Z:\zeta(x)>\xi(x),\zeta(x+1)<\xi(x+1)\right\}\fstop
\end{equation*}
Using the fact that any pair counted in $D_{N}$ is annihilated with
rate at least $1$, and discrepancies only enter from the boundary,
$\cL E_{N}\le-D_{N}+C$, for $C>0$ not depending on $N$. Since at stationarity we have
$\E[\cL E_{N}]=0$,  translation invariance yields
\begin{equation*}
\E[D_{N}]=N\P(\zeta(0)>\xi(0),\zeta(1)<\xi(1))\le C\comma
\end{equation*}
which is possible only if $\P(\zeta(0)>\xi(0),\zeta(1)<\xi(1))=0$.

Moreover, the probability, for any $k\ge2$, to have discrepancies
of opposite sign at distance $k$ must vanish: since discrepancies
move one step to the right or to the left with rate at least $1$,
any pair of opposite sign discrepancies at distance $k$ produces
with rate at least $2$ a pair of opposite sign discrepancies at distance
$k-1$. By induction we conclude that no such pair could exist.
As a consequence, discrepancies are either all positive or all are negative, see, e.g., \cite[Chapter VIII. Lemma 3.2]{liggett_interacting_2005}.

Now, let $\pi\in(\cI\cap\cS)_e$, and set
$\rho\eqdef\pi(\eta(0)=1)$. Take a stationary translation-invariant coupling
$\lambda$ of $\pi$ and $\nu_\rho$,  and set
$A\eqdef\{\zeta\ge\xi\}$ and $B\eqdef\{\zeta\le\xi\}$. By the previous
argument, $\lambda(A\cup B)=1$. Monotonicity gives
$\car_A(\zeta_t,\xi_t)\ge\car_A(\zeta_0,\xi_0)$. Since $\lambda$ is
stationary, the two sides have the same expectation and hence are equal
almost surely. Consequently, the conditional laws $\lambda(\,\cdot\,|A)$ and $\lambda(\,\cdot\,|A^\complement)$ of $\lambda$ on $A$ and
$A^\complement$, respectively, are stationary under the coupling; they are also translation invariant, since
$A$ is shift-invariant.
 If $0<\lambda(A)<1$, let $\pi_A$ and
$\pi_{A^\complement}$ be the first marginals of these conditional laws. Then
$\pi=\lambda(A)\pi_A+(1-\lambda(A))\pi_{A^\complement}$, with
$\pi_A,\pi_{A^\complement}\in\cI\cap\cS$. The extremality of $\pi$ in
$\cI\cap\cS$ therefore gives
$\pi_A=\pi_{A^\complement}=\pi$. The same decomposition applied to the second
marginal, together with the shift ergodicity of $\nu_\rho$, shows that both
conditional second marginals equal $\nu_\rho$.
Under $\lambda(\,\cdot\,|A)$, the random variable
$\zeta(0)-\xi(0)$ is nonnegative and has expectation $\rho-\rho=0$; hence, it
vanishes almost surely. By translation invariance, $\zeta=\xi$ almost surely
on $A$. Since $\lambda(A\cup B)=1$, one has $\zeta\le\xi$ almost surely on
$A^\complement$, and the same argument applies to $\xi-\zeta$. The cases
$\lambda(A)\in\{0,1\}$ follow directly. Thus $\zeta=\xi$ $\lambda$-a.s., and
therefore $\pi=\nu_\rho$.
\end{proof}

\section{Entropy solutions and proofs of limit theorems}\label{sec:proofs}

This section contains the proofs of the results stated in Section \ref{sec:hydro}.
As outlined in Section \ref{sec:intro}, to prove hydrodynamics we rely on the constructive method 
developed in 
\cite{bahadoran_constructive_2002,bahadoran_euler_2006,bgrs3,bgrs5} for
one-dimensional conservative attractive particle systems
under a hyperbolic space-time scaling. We will explain
how and why this constructive method can be applied, and  give details only for the specific results 
and computations  needed to
apply it to our model. We chose to concentrate on the adaptation of the results in 
\cite{bahadoran_constructive_2002}, since they deal with models with product invariant measures,
which is the case of ${\rm DEP}$.
 We refer to 
\cite{bgrs5} for an overview of results derived through this constructive approach, and of
 models to which it can be applied. In view of the results that we derived in Section \ref{sec:coupling}, 
${\rm DEP}$ is an example close to the models fitting the general presentation in \cite[Section 6]{bgrs5}. 	

We specialize our discussion to ${\rm DEP}$'s flux $G_\DEP$ given in \eqref{eq:HDL}. Note that $G_\DEP$ is smooth. For notational convenience, we write $G=G_\DEP$  all throughout. In what follows, for  any open $A\subset \R^d$, $d\ge 1$, and integer $k\ge 1$, we write $\cC^k(\overline{A})$ for the space of $k$-differentiable functions on $A$, with all derivatives continuously extendable up to the boundary (if $\partial A\neq \emp$); and
$\cC_c^k(\overline{A})$ indicates its subspace of  compactly supported functions. 

\subsection{Entropy solutions}\label{sec:entropy-solutions}For the reader's convenience, let us recall some classical definitions and facts about one-dimensional scalar equations 
(see, e.g., \cite{ballou_solutions_1970} or \cite[Section 2]{serre_systems_1999}). 
 This presentation relies on \cite{bahadoran_constructive_2002}, \cite[Section 2.2]{bahadoran_euler_2006}, \cite[Section 4]{bgrs5}, and the references therein.

A measurable bounded function $u:\R\times [0,\infty)\to \R$ is a \textit{weak solution} to the Cauchy problem 
\begin{equation}\label{eq:cauchy}
	\begin{dcases}
		\partial_t u + \partial_xG(u)=0\\
		u(\emparg,0)=u_0 \comma
	\end{dcases}
\end{equation}
associated to \eqref{eq:HDL}
if the following holds true: for all  $\varphi\in \cC_c^1(\R\times [0,\infty))$, 
\begin{equation*}
	\int_0^\infty \int_\R \tonde{u\,\partial_t \varphi+ G(u)\,\partial_x \varphi}\dd x\dd t + \int_\R u_0(x)\,\varphi(x,0)\,\dd x=0\fstop
\end{equation*}
A weak solution $u$ to
\eqref{eq:cauchy}
is an \emph{entropy solution} if  the following entropy inequality holds true: for all $\varphi\in \cC_c^1(\R\times [0,\infty))$, $\varphi\ge0$, and entropy--entropy-flux pair $(E,F)$ associated to the flux $G$ (i.e., $E\in \cC^2(\R)$ is convex, $F\in \cC^1(\R)$, and $F'=E'G'$)
\begin{equation*}
	\int_0^\infty	\int_\R \tonde{E(u)\,\partial_t \varphi+ F(u)\,\partial_x \varphi}\dd x\dd t + \int_\R E(u_0(x))\, \varphi(x,0)\, \dd x\ge 0\fstop
\end{equation*}

 A necessary and sufficient condition for a piecewise smooth function $u$ to be a weak solution to equation \eqref{eq:cauchy}
is that: (a) $u$ solves \eqref{eq:cauchy} at points of smoothness; (b) if $x(t)$ is a curve of discontinuity of the solution, then the \textit{Rankine-Hugoniot condition}
\begin{equation}\label{Rankine-Hugoniot}
\dot x(t)=\frac{G(u^-)-G(u^+)}{u^--u^+}=:S[u^+;u^-]
\end{equation}
holds along $x(t)$ for a.e.\ $t>0$, where $u^\pm:=u(x(t)^\pm,t)=\lim_{h\downarrow 0}u(x(t)\pm h,t)$.

To ensure uniqueness, \textit{Ole\u{\i}nik's entropy condition} is sufficient: 
a discontinuity $(u^+,u^-)$ (where $u^\pm:=u(x^\pm ,t)$, for some $x\in \R$ and $t>0$) is an \textit{entropy shock} if and only if:

\textit{The chord of the graph of $G$ between  $u^-$ and $u^+$ 
lies below the graph if $u^-<u^+$, above the graph if $u^->u^+$.}  

\begin{proposition}(\cite[Proposition 2.2]{bahadoran_euler_2006})\label{prop:bgrs2-entropy-sol}
A weak solution $u$ to
\eqref{eq:cauchy} with (locally, uniformly over time) bounded space variation 
is an \emph{entropy solution}  if and only if, for a.e.\ $t>0$, all discontinuities of $u(\emparg,t)$ are entropy shocks.
\end{proposition}

We start by considering Riemann initial data,  relying on Proposition \ref{prop:bgrs2-entropy-sol} 
to select the entropy solution among the weak ones and to determine it explicitly.

\subsection{Riemann case}\label{sec:riemann} 
When dealing with step (or  \emph{Riemann}) initial conditions, i.e., 
	\begin{equation}\label{eq:initial-riemann}
	u_0=\lambda\car_{(-\infty,0)}+\rho \car_{[0,+\infty)}\comma\qquad \text{for some}\ \lambda, \rho \in [0,1]\comma  
\end{equation}
we look for \emph{self-similar weak solutions} $u(x,t)$ to equation \eqref{eq:cauchy} in the following form:
\begin{equation*}
	u(x,t)= u(x/t,1)\equiv u(v,1)\comma\qquad v=x/t\fstop
\end{equation*}
This suffices because of the invariance of both equation and initial condition under the  scaling $(x,t)\mapsto (ax,at)$, $a>0$. 
\par\medskip

The flux $G=G_\DEP$ given in \eqref{eq:HDL} satisfies 
\begin{equation*}
	H(u)\eqdef G'(u)=1-12\tonde{u-\frac12}^2\comma\qquad G''(u)=24\tonde{\frac12-u}\fstop
\end{equation*}
Hence, $G$ is strictly convex (resp.\ concave) for $u<1/2$ (resp.\ $u>1/2$), with a single inflection  point at $u=1/2$. Therefore, 
\cite[Proposition 2.1]{bahadoran_constructive_2002}, that we now quote, applies.	

\begin{proposition}\label{prop:bgrs2-prop2.1}(\cite[Proposition 2.1]{bahadoran_constructive_2002}).
For a flux $G\in\cC^2(\R)$, the self-similar entropy weak solution $u(v,1)$ of equation \eqref{eq:cauchy} with $\lambda\le \rho$ (resp.\ $\lambda \ge \rho$)
is the unique
global minimum (resp.\ maximum) of $G(s)-vs$ at its  points of continuity. 
\end{proposition}

The explicit construction of entropy solutions follows by \emph{Step 2} in \cite[Section 2.1]{bahadoran_constructive_2002}, which we now briefly sketch. 
 
The characteristic speed $[0,1]\ni u\mapsto H(u)$ takes values in $[-2,1]$. Its  inverse branches read, for $v\in [-2,1]$, as
\begin{align*}
	H_{<1/2}^{-1}(v)=\frac12-\sqrt{\frac{1-v}{12}}\comma\qquad H_{>1/2}^{-1}(v)=\frac12+\sqrt{\frac{1-v}{12}}\fstop
\end{align*}

Let $G_*^u$ denote the lower convex envelope of $G$
on the interval $(-\infty,u]$, while $G_u^*$ denotes
the upper concave envelope of $G$ on $[u,+\infty)$.
For $u<1/2$, let $u^*=u^*(u)>1/2$ be the smallest
point strictly larger than $u$ where $G_u^*$ coincides
with $G$; for $u>1/2$, similarly define
$u_*=u_*(u)<1/2$ as the largest point strictly smaller
than $u$ where $G_*^u$ coincides with $G$.
Both points are determined by solving
$G'(a)=(G(a)-G(u))/(a-u)$ with $a\ne u$.
For our flux $G=G_\DEP$, this gives
\begin{equation*}
	u^*(u)=\frac34-\frac{u}{2}
	\comma \text{for}\ u<\frac12\comma
	\qquad
	u_*(u)=\frac34-\frac{u}{2}
	\comma \text{for}\ u>\frac12\fstop
\end{equation*}

We find entropy solutions for the case
$\rho\le 1/2$; the case	 $\rho>1/2$ may be dealt with analogously and, thus, is left to the reader. 
\begin{enumerate}
	\item If $\lambda\le \rho$, the relevant part of the flux $G$ is convex; thus, $H(\lambda)<H(\rho)$, and the unique entropy solution is 
	the (continuous) \textit{rarefaction fan} (Figure \ref{fig:rarefaction}):	
	\begin{equation*}
		u(x,t)=u(x/t,1)= \begin{dcases}
			\lambda &\text{if}\ x/t\le H(\lambda)\\
			H_{<1/2}^{-1}(x/t) &\text{if}\ H(\lambda)<x/t<H(\rho)\\
			\rho &\text{if}\ x/t\ge H(\rho)\fstop
		\end{dcases}
	\end{equation*}
	\item If $\lambda >\rho$, we further distinguish two cases:
	\begin{enumerate}\item  If $\lambda \le\rho^*=3/4-\rho/2$, we have $H(\lambda)>H(\rho)$;	 
	then, the unique entropy solution is  the \textit{shock} (Figure \ref{fig:shock}): 
		\begin{equation*}
			u(x,t)=u(x/t,1)=\begin{dcases}
				\lambda &\text{if}\ x/t< S[\lambda;\rho]
				\\
				\rho &\text{if}\ x/t>S[\lambda;\rho]\comma
			\end{dcases}
		\end{equation*}
		where $S[\lambda;\rho]$ is identified by the Rankine-Hugoniot condition  \eqref{Rankine-Hugoniot}.
		\item  If $\lambda > \rho^*=3/4-\rho/2$, 	we have $H(\lambda)\le H(\rho^*)$; 
		hence, the entropy solution is a mixed one, namely,  a rarefaction fan followed by a shock 
		(Figure \ref{fig:rarefaction-shock}); this is called  
		a \emph{contact discontinuity} in \cite{ballou_solutions_1970}:
		\begin{equation*}
			u(x,t)=u(x/t,1)=\begin{dcases}
				\lambda &\text{if}\ x/t\le H(\lambda)\\
				H_{>1/2}^{-1}(x/t) &\text{if}\ H(\lambda)<x/t\le H(\rho^*)\\
				\rho &\text{if}\ x/t> H(\rho^*)\fstop
			\end{dcases}
		\end{equation*}	
	\end{enumerate}
\end{enumerate}
	\begin{figure}[t]
	\centering
	\includegraphics[width=1\textwidth]{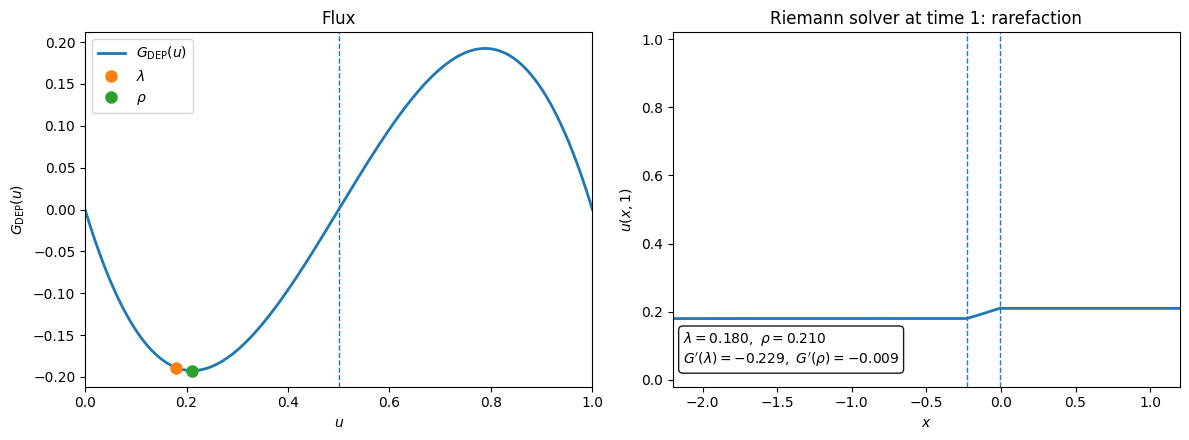}
	\caption{A rarefaction solution at time $t=1$.}
	\label{fig:rarefaction}
\end{figure} 
\begin{figure}[t]
	\centering
	\includegraphics[width=1\textwidth]{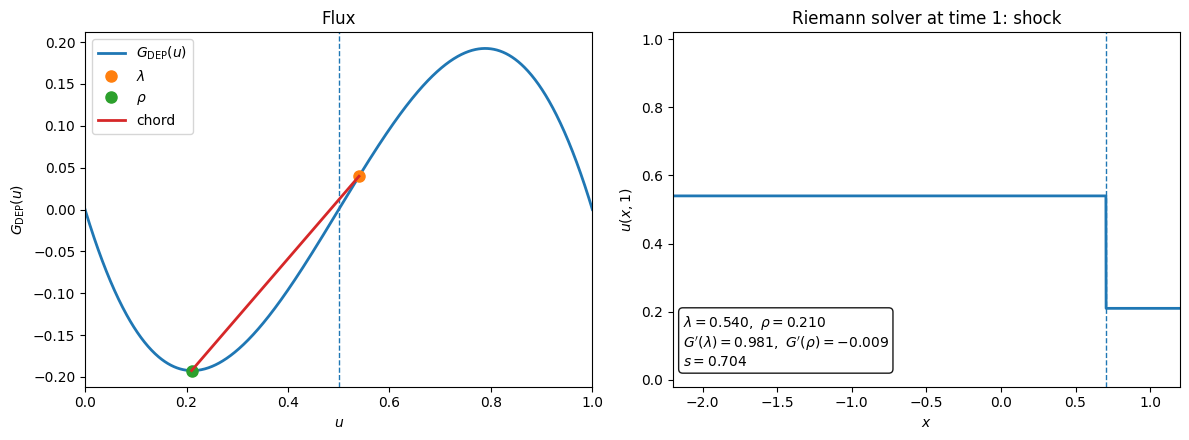}
	\caption{A shock solution at time $t=1$.}
	\label{fig:shock}
\end{figure} 
\begin{figure}[t]
	\centering
	\includegraphics[width=1\textwidth]{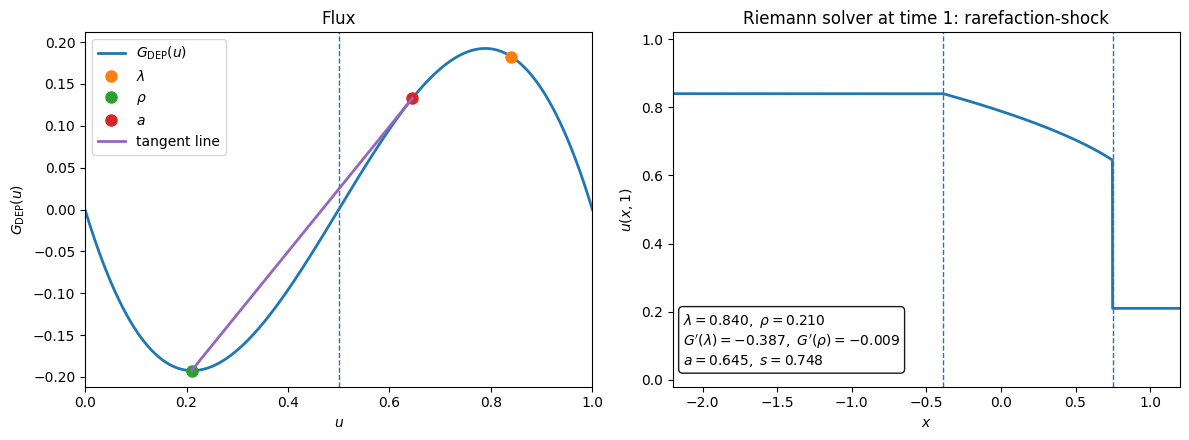}
	\caption{A rarefaction-shock solution at time $t=1$.}
	\label{fig:rarefaction-shock}
\end{figure}

\subsection{Conservation of local equilibrium from Riemann profiles}\label{subsec:loc-eq-riemann}

We derive the following result,  corresponding to \cite[Theorem 2.1]{bahadoran_constructive_2002}. 
\begin{theorem}[Conservation of local equilibrium --- Riemann case]\label{th:local-eq-riemann}
	Let,  for some $\lambda, \rho\in [0,1]$, $\mu_{\lambda,\rho}$ be the product measure on $\X$ 
	associated to $u_0=\lambda\car_{(-\infty,0)}+\rho\car_{[0,+\infty)}$ (given in \eqref{eq:initial-riemann}). 
	Then, letting $u(\emparg,\emparg):\R\times [0,+\infty)\to [0,1]$ denote the entropy solution associated to $u_0$ 
	(cf.\ Section \ref{sec:riemann}), we have, for all $t\ge 0$,
	\begin{equation*}
		\lim_{\eps\to 0} \tau_{\lfloor x\eps^{-1}\rfloor}\tttonde{\mu_{\lambda,\rho} \cP_{t\eps^{-1}}} 
		= \nu_{u(x,t)}\comma\qquad \text{for every continuity point $x\in \R$ of $u(\emparg,t)$}\fstop
	\end{equation*}
\end{theorem}
\begin{proof}
We follow the steps in \cite[Section 2.2]{bahadoran_constructive_2002} 
(see also \cite[Section 3]{andjel_vares_hydrodynamic_1987}),   that apply here without any change. We now summarize them.
The first step of the proof
is to show that a weak Ces\'aro limit of the measure of the process belongs to $\cI\cap \cS$. 
The second step is a computation of the Ces\'aro limiting density inside a macroscopic box. 
Both steps rely on attractiveness and on the characterization of $\cI\cap \cS$.
Let us now quote these two results:

\begin{lemma}\label{lem:from-bgrs1-and-av}
(\cite[Lemmas 2.3, 2.4]{bahadoran_constructive_2002},  \cite[Lemmas 3.1, 3.2]{andjel_vares_hydrodynamic_1987}).
Let $\mu$ be a probability measure on $\X$ such that: 
\begin{enumerate}[(a)]
	\item $\nu_\rho\leq \mu\leq\nu_\lambda$ for some $0\leq\rho<\lambda$; 
	\item either $\tau_1\mu\leq\mu$ or $\tau_1\mu\geq\mu$.
\end{enumerate}

Then, any sequence $T_n\to\infty$ has a subsequence $T_{n_m}$ for which there exists
a dense countable subset $D$ of $\R$ satisfying
\begin{equation*}
\lim_{m\to\infty}\frac{1}{T_{n_m}}\int_0^{T_{n_m}}(\tau_{\lfloor vt\rfloor}\mu)\cP_t\, \dd t=\int\nu_\alpha\,\gamma_v(\dd\alpha)=\mu_v\in\cI\cap \cS\comma\qquad v \in D\comma
\end{equation*} 
where $\gamma_v$ is a probability measure on $[\rho,\lambda]$. 
Also, if $u<v$ are in $D$,
\begin{equation}\label{eq:limiting-density}
\lim_{m\to\infty}\mu\cP_{T_{n_m}}\left(\frac{1}{T_{n_m}}\sum_{x= \lfloor uT_{n_m}\rfloor }^{\lfloor vT_{n_m}\rfloor }\eta(x)\right)=F(v)-F(u)\comma
\end{equation}
with, for $w\in D$, $F(w)=\int\quadre{w\alpha-G(\alpha)}\gamma_w(\dd\alpha)$. 
\end{lemma}
 Note that the macroscopic flux $G=G_\DEP$ given in \eqref{eq:HDL}
appears in the function $F$ in \eqref{eq:limiting-density}.  The idea to derive \eqref{eq:limiting-density}
is the following. Let, for $u<v$ elements of $D$ and $t\ge 0$,
\begin{equation*}
\Psi(t)\eqdef\sum_{x= \lfloor ut\rfloor+1}^{\lfloor vt\rfloor}\mu\cP_t(\eta(x)) + \tonde{vt - \lfloor vt\rfloor}\mu\cP_t(\eta(\lfloor vt\rfloor+1)) + \tonde{\lfloor ut\rfloor+1-ut}\mu\cP_t(\eta(\lfloor ut\rfloor))\fstop
\end{equation*}
Then, we take the limit in $m$ of ${\Psi(T_{n_m})}/{T_{n_m}} =\mu(\eta(0))/{T_{n_m}}+T_{n_m}^{-1}\int_0^{T_{n_m}} \Psi'(t)\,\dd t$, where, if $ut,vt\notin\Z$,
\begin{equation}\label{Psiprime}
	\begin{aligned}
\Psi'(t)&= v\,\mu\cP_t(\eta(\lfloor vt\rfloor+1)) -u\, \mu\cP_t(\eta(\lfloor ut\rfloor)) +
\sum_{\lfloor ut\rfloor+1}^{\lfloor vt\rfloor}\mu\cP_t(\cL \eta(x))\\
& + \tonde{vt - \lfloor vt\rfloor}\mu\cP_t(\cL\eta(\lfloor vt\rfloor+1)) + \tonde{\lfloor ut\rfloor+1-ut}\mu\cP_t(\cL\eta(\lfloor ut\rfloor))\fstop
\end{aligned}
\end{equation}
The first two terms in \eqref{Psiprime} arise from the motion of the
endpoints, whereas the three terms containing $\cL\eta$ describe the evolution
of the occupation variables. Writing $j_x=j_x(\eta)$ for the microscopic current
through $x$, so that $\cL\eta(x)=j_{x-1}-j_x$ (cf.\ \eqref{eq:jeta}), these three terms telescope
according to
\begin{align*}
	&\sum_{x=\lfloor ut\rfloor+1}^{\lfloor vt\rfloor}\cL\eta(x)
	+\tonde{vt-\lfloor vt\rfloor}\cL\eta(\lfloor vt\rfloor+1)
	+\tonde{\lfloor ut\rfloor+1-ut}\cL\eta(\lfloor ut\rfloor)
	\\
	&=j_{\lfloor ut\rfloor-1}-j_{\lfloor vt\rfloor} + \tonde{ut-\lfloor ut\rfloor }\tttonde{j_{\lfloor ut\rfloor}-j_{\lfloor ut\rfloor-1}}-\tonde{vt-\lfloor vt\rfloor}\tttonde{j_{\lfloor vt\rfloor+1}-j_{\lfloor vt\rfloor}}\fstop
\end{align*}
While the last two terms on the right-hand side above vanish, the remaining expression $j_{\lfloor ut\rfloor -1}-j_{\lfloor vt\rfloor}$ represents the current entering through the left moving endpoint
minus the current leaving through the right one. Its time average converges
to $\int G(\alpha)\,\gamma_u(\dd\alpha)
-\int G(\alpha)\,\gamma_v(\dd\alpha)$, while those of the first two terms in
\eqref{Psiprime} converge to
$v\int\alpha\,\gamma_v(\dd\alpha)
-u\int\alpha\,\gamma_u(\dd\alpha)$. Their sum is precisely $F(v)-F(u)$. For the details, see \cite[Lemma 3.2]{andjel_vares_hydrodynamic_1987} and \cite[Lemma 7.3]{amir_hydrodynamics_2025}. 
\par\smallskip

The third (and main) step of the proof, which consists in proving that $\gamma_v$ is the Dirac measure concentrated on $u(v,1)$, 
relies on Proposition \ref{prop:bgrs2-prop2.1}. The last step is to prove
that Ces\'aro limits are actually weak limits; this is proved via monotonicity arguments.
\end{proof}

\subsection{From Riemann to general initial profiles}\label{subsec:cauchy}

For existence and uniqueness of entropy solutions to \eqref{eq:cauchy} with general, 
nonnegative, and bounded initial data, we refer to, e.g., \cite[Section 5]{serre_systems_1999},  
\cite[Theorem 3.1]{bahadoran_constructive_2002}, and references therein.	

We now outline the proof of Theorem \ref{th:HDL}, that is, the derivation of hydrodynamics in the Cauchy case.

\begin{proof}[Proof of Theorem \ref{th:HDL}]

	The main result in \cite[Section 3]{bahadoran_constructive_2002} is the hydrodynamic limit 
	from general initial conditions (\cite[Theorem 3.2]{bahadoran_constructive_2002}), 
 derived through an approximation scheme for the profiles as well as for the configurations, 	
	and its proof 
	fully adapts to our setting because: on the one hand, we already proved in Corollary \ref{cor:macrostab} 
	that ${\rm DEP}$ is macroscopically stable; on the other hand, \cite[Theorem 3.1]{bahadoran_constructive_2002} 
	on regularity properties of the macroscopic entropy solutions holds true in our case. Moreover, note that 
	\cite[Lemma 3.1]{bahadoran_constructive_2002} (that is, finite propagation property)
	and \cite[Lemma 3.2]{bahadoran_constructive_2002}
	are proved for bounded jump rates and for finite-range jumps and interactions, thus, covering the example of ${\rm DEP}$. 
	This proves our Theorem \ref{th:HDL}.

	Let us explain in a few words how the approximation scheme works.
	First, we restrict to finite initial configurations $\eta_0^\eps$,
	with bounded macroscopic support, associated to a compactly
	supported initial density profile $u_0$ with finite variation.
	The extension to arbitrary bounded measurable initial profiles
	is then obtained by approximation, macroscopic stability,
	and finite propagation.
	
	The distances used to control the approximation are defined,
	for integrable $[0,1]$-valued density profiles $u,v$ on $\R$,
	and configurations $\eta,\xi$ with finitely many particles, by
	\begin{align*}
		\varDelta(u,v)
		&\eqdef
		\sup_{x\in\R}
		\tabs{
			\int_{-\infty}^x
			\tttonde{u(y)-v(y)}\,\dd y
		}\comma\\
		\varDelta^\eps(\eta,\xi)
		&\eqdef
		\eps\sup_{x\in\Z}
		\tabs{
			\sum_{y\leq x}
			\tttonde{\eta(y)-\xi(y)}
		}\comma
		\qquad\text{(cf.\ Corollary \ref{cor:macrostab})}\\
		\varDelta^\eps(\eta,u)
		&\eqdef
		\sup_{x\in\Z}
		\tabs{
			\eps\sum_{y\leq x}\eta(y)
			-
			\int_{-\infty}^{x\eps}u(y)\,\dd y
		}\fstop
	\end{align*}
	In particular, the microscopic--macroscopic error at time $t$
	is
	\begin{equation*}
		\mathscr D^\eps(t)
		\eqdef
		\varDelta^\eps
		\tttonde{
			\eta^\eps_{t\eps^{-1}},
			u(\emparg,t)
		}\comma
	\end{equation*}
	where $\eta^\eps_{t\eps^{-1}}$ is the original configuration
	evolved up to the macroscopic time $t$, and $u(\emparg,t)$
	is the entropy solution with initial condition $u_0$.
	
	The approximation scheme consists in comparing these two
	evolutions, over short time intervals, with auxiliary ones
	starting from piecewise constant profiles. Fix $t\geq0$
	and $\delta>0$. In the {\em approximation step}, replace
	$u(\emparg,t)$ by a piecewise constant profile
	$\widetilde u_t$, whose steps have length at least
	$\Delta x$ and such that
	\begin{equation*}
		\varDelta
		\tttonde{u(\emparg,t),\widetilde u_t}
		\leq\delta\Delta x\fstop
	\end{equation*}
	Such an approximation exists for $\Delta x>0$ sufficiently
	small, by \cite[Lemma 3.6]{bahadoran_constructive_2002}.
	Choose also auxiliary configurations $\widetilde\eta_t^\eps$
	with bounded support and density profile $\widetilde u_t$.
	In the {\em evolution step}, evolve the auxiliary pair
	during a macroscopic time interval of length $\Delta t$.
	Denote the resulting configuration and entropy profile by
	$\widetilde\eta_{t,\Delta t}^\eps$ and
	$\widetilde u_t(\emparg,\Delta t)$, respectively.
	
	Choose $\Delta t$ proportional to $\Delta x$, with
	\begin{equation*}
		\frac{\Delta t}{\Delta x}
		<
		\frac{1}{2v_\star}\comma\qquad
		v_\star\eqdef\max\{v,\ttnorm{G'}_\infty\}\comma
	\end{equation*}
	where $v>0$ is the microscopic finite-propagation speed
	from Proposition \ref{prop:finite-prop}. On this time
	interval, the Riemann waves starting from consecutive
	discontinuities of $\widetilde u_t$ do not interact.
	By finite propagation, the hydrodynamic limit of the
	auxiliary evolution is therefore obtained by superposing
	the corresponding Riemann solutions. In particular,
	\begin{equation*}
		\varDelta^\eps
		\tttonde{
			\widetilde\eta_{t,\Delta t}^\eps,
			\widetilde u_t(\emparg,\Delta t)
		}
		=o_\eps(1)\comma
	\end{equation*}
	where $o_\eps(1)$ denotes a random quantity converging
	to zero in probability, for fixed approximation parameters.
	
	The auxiliary evolution is then compared with the original
	one. By the triangle inequality, exact macroscopic stability,
	and contraction of the entropy semigroup with respect to
	$\varDelta$, we obtain
	\begin{align*}
		\mathscr D^\eps(t+\Delta t)
		&\leq
		\varDelta^\eps
		\tttonde{
			\eta^\eps_{(t+\Delta t)\eps^{-1}},
			\widetilde\eta_{t,\Delta t}^\eps
		}
		+
		\varDelta^\eps
		\tttonde{
			\widetilde\eta_{t,\Delta t}^\eps,
			\widetilde u_t(\emparg,\Delta t)
		}\\
		&\quad+
		\varDelta
		\tttonde{
			\widetilde u_t(\emparg,\Delta t),
			u(\emparg,t+\Delta t)
		}\\
		&\leq
		\varDelta^\eps
		\tttonde{
			\eta^\eps_{t\eps^{-1}},
			\widetilde\eta_t^\eps
		}
		+
		\varDelta
		\tttonde{
			\widetilde u_t,
			u(\emparg,t)
		}
		+o_\eps(1)\\
		&\leq
		\mathscr D^\eps(t)+2\delta\Delta x+o_\eps(1)\fstop
	\end{align*}
	The same comparison can be repeated at subsequent times,
	always starting from the original microscopic configuration
	and entropy profile.
	
	Since $\Delta t$ is proportional to $\Delta x$, the
	approximation error at each step is an arbitrarily small
	multiple of the time increment. More precisely, the
	time-continuity estimates of
	\cite[Lemma 3.2]{bahadoran_constructive_2002},
	together with continuity of the entropy semigroup,
	imply that subsequential limits of $\mathscr D^\eps$ have
	continuous paths. Passing to such limits in the above
	inequality, with $\Delta x$ sufficiently small depending
	on $t$ and $\delta$, shows that their increase over
	any time interval is bounded by a constant times
	$\delta$ times the length of the interval.
	Since $\mathscr D^\eps(0)\to0$ in probability and $\delta>0$
	is arbitrary, every subsequential limit vanishes.
	Hence, $\mathscr D^\eps(t)\to0$ in probability, concluding the
	proof.
\end{proof}

\section{Strong hydrodynamic limit}\label{sec:strong-hydro}
We conclude this article by mentioning that we also have a strong hydrodynamic 
limit for ${\rm DEP}$ that we now state. 

Indeed, thanks to the graphical representation outlined in Section \ref{sec:graphical}, we construct infinitely many copies of ${\rm DEP}$
on the probability space
$(\Omega_0\times\Omega, \cF_0\otimes\cF, \P_0\otimes\P)$, where  
$(\Omega_0, \cF_0, \P_0)$ is a probability space used for the random initial states, 
and $(\Omega, \cF,\P)$ is a Poisson space used to construct the evolution from a given state. In what follows, we write $\widetilde \P=\P_0\otimes \P$.

\begin{theorem}[Strong hydrodynamic limit]\label{th:strong-HDL}
	Let $(\eta_0^n)_{n\in \N^*}\in \Omega_0$ be a sequence of $\X$-valued random variables with strong density profile $u_0$, i.e., $u_0:\R\to [0,1]$ is measurable and,  $\P_0$-a.s., one has
	\begin{equation*}
		\frac1n\sum_{x\in \Z} f(\tfrac{x}n)\,\eta_0^n(x)
		\xrightarrow{n\to \infty}\int_{\R} f(x)\,u_0(x)\,\dd x\comma\qquad  f \in \cC_c(\R)\fstop
	\end{equation*}
	Define, as in Section \ref{sec:graphical}, $\eta_t^n=\eta(\eta_0^n,\emparg,t)$, $n\in \N^*$ (i.e., $\eta(\eta_0^n,\emparg,0)=\eta_0^n$ for each $n\in \N^*$, but employing a common set of Poisson clocks). 
	Then,  letting $u(\emparg,\emparg):\R\times [0,\infty)\to [0,1]$ 
	be the  
	entropy solution to \eqref{eq:HDL} with initial condition $u_0$, we have, $\widetilde\P$-a.s.,
	\begin{equation*}
	\sup_{t \in [0,T]}\abs{	\frac1n\sum_{x\in \Z} f(\tfrac{x}n)\,\eta_{tn}^n(x)-\int_{\R} f(x)\,u(x,t)\, \dd x}\xrightarrow{n\to\infty}0\comma\qquad T>0\comma f \in \cC_c(\R)\fstop	\end{equation*}
\end{theorem}   

The proof of  this theorem follows the lines of \cite{bgrs3}:
 we still have to first consider the Riemann case, then to go to the general case using an approximation scheme.
The main change with the previous approach is that now currents become the central object to deal with.
To solve the Riemann problem, we combine proofs of almost sure analogues for currents of the results of 
\cite{andjel_vares_hydrodynamic_1987, bahadoran_constructive_2002, bahadoran_euler_2006}
  with a space-time ergodic
theorem for particle systems and with large deviation estimates for the empirical measure. 
In the approximation steps, we need estimates uniform in time, and each approximation step requires
a control with exponential bounds. 
For further details, we refer to \cite{bgrs3}.

\subsection*{Acknowledgments} F.S.\ thanks MAP5 lab for hospitality and financial support. While this work was written, the same author was associated to INdAM (Istituto
Nazionale di Alta Matematica “Francesco Severi”), the group GNAMPA, and the GNAMPA-INdAM project “Stochastic exchange models: from kinetic theory to opinion dynamics”.
We would also like to thank Giuseppe Cannizzaro for interesting discussions,  and the two anonymous referees for their detailed comments, which helped us
improve the presentation and clarify several points in the proofs.


\begin{thebibliography}{PBMH92}
	
	\bibitem[ABBS25]{amir_hydrodynamics_2025}
	Gideon Amir, Christophe Bahadoran, Ofer Busani, and Ellen Saada.
	\newblock Hydrodynamics and relaxation limit for multilane exclusion process
	and related hyperbolic systems.
	\newblock {\em arXiv:2501.19355}, 2025.
	
	\bibitem[AV87]{andjel_vares_hydrodynamic_1987}
	Enrique~Daniel Andjel and Maria~Eul\'{a}lia Vares.
	\newblock Hydrodynamic equations for attractive particle systems on {${\bf
			Z}$}.
	\newblock {\em J. Statist. Phys.}, 47(1-2):265--288, 1987.
	
	\bibitem[Bal70]{ballou_solutions_1970}
	Donald~P. Ballou.
	\newblock Solutions to nonlinear hyperbolic {C}auchy problems without convexity
	conditions.
	\newblock {\em Trans. Amer. Math. Soc.}, 152:441--460 (1971), 1970.
	
	\bibitem[BGRS02]{bahadoran_constructive_2002}
	C.~Bahadoran, H.~Guiol, K.~Ravishankar, and E.~Saada.
	\newblock A constructive approach to {E}uler hydrodynamics for attractive
	processes. {A}pplication to {$k$}-step exclusion.
	\newblock {\em Stochastic Process. Appl.}, 99(1):1--30, 2002.
	
	\bibitem[BGRS06]{bahadoran_euler_2006}
	C.~Bahadoran, H.~Guiol, K.~Ravishankar, and E.~Saada.
	\newblock Euler hydrodynamics of one-dimensional attractive particle systems.
	\newblock {\em Ann. Probab.}, 34(4):1339--1369, 2006.
	
	\bibitem[BGRS10]{bgrs3}
	Christophe Bahadoran, Herv{\'e} Guiol, Krishnamurthi Ravishankar, and Ellen
	Saada.
	\newblock Strong hydrodynamic limit for attractive particle systems on
	{{$\mathbb{Z}$}}.
	\newblock {\em Electron. J. Probab.}, 15:1--43, 2010.
	\newblock Id/No 1.
	
	\bibitem[BGRS19]{bgrs5}
	Christophe Bahadoran, Herv{\'e} Guiol, Krishnamurthi Ravishankar, and Ellen
	Saada.
	\newblock Constructive {Euler} hydrodynamics for one-dimensional attractive
	particle systems.
	\newblock In {\em Sojourns in probability theory and statistical physics. III.
		Interacting particle systems and random walks, a festschrift for Charles M.
		Newman}, pages 43--89. Singapore: Springer; Shanghai: NYU Shanghai, 2019.
		
	\bibitem[BKS85]{vanbeijeren_kutner_spohn_excess_1985}
	H.~van Beijeren, R.~Kutner, and H.~Spohn.
	\newblock Excess noise for driven diffusive systems.
	\newblock {\em Phys. Rev. Lett.}, 54(18):2026--2029, 1985.
	
	\bibitem[BPB94]{binder1994scaling}
	P.-M. Binder, M.~Paczuski, and Mustansir Barma.
	\newblock Scaling of fluctuations in one-dimensional interface and hopping
	models.
	\newblock {\em Phys. Rev. E}, 49(2):1174, 1994.
	
	\bibitem[CDR16]{crawford2016invariance}
	Nick Crawford and Wojciech De~Roeck.
	\newblock Invariance principle for \textquoteleft push\textquoteright\ tagged
	particles for a {T}oom interface.
	\newblock {\em arXiv:1610.07765}, 2016.
	
	\bibitem[CET23]{cannizzaro_erhard_toninelli_stationary_2023}
	Giuseppe Cannizzaro, Dirk Erhard, and Fabio Toninelli.
	\newblock The stationary {AKPZ} equation: logarithmic superdiffusivity.
	\newblock {\em Comm. Pure Appl. Math.}, 76(11):3044--3103, 2023.
	
	\bibitem[CK20]{crawford_kozma_toom_2020}
	Nicholas Crawford and Gady Kozma.
	\newblock The {T}oom interface via coupling.
	\newblock {\em J. Stat. Phys.}, 179(2):408--447, 2020.
	
	\bibitem[DLSS91]{derrida_dynamics_1991}
	B.~Derrida, J.~L. Lebowitz, E.~R. Speer, and H.~Spohn.
	\newblock Dynamics of an anchored {T}oom interface.
	\newblock {\em J. Phys. A}, 24(20):4805--4834, 1991.
	
	\bibitem[DMP91]{de_masi_mathematical_1991}
	Anna De~Masi and Errico Presutti.
	\newblock {\em Mathematical methods for hydrodynamic limits}, volume 1501 of
	{\em Lecture Notes in Mathematics}.
	\newblock Springer-Verlag, Berlin, 1991.
	
	\bibitem[DS92]{devilliard_spohn_universality_1992}
	P.~Devillard and H.~Spohn.
	\newblock Universality class of interface growth with reflection symmetry.
	\newblock {\em J. Statist. Phys.}, 66(3-4):1089--1099, 1992.
	
	\bibitem[Dur95]{durrett_saint-flour}
	Rick Durrett.
	\newblock Ten lectures on particle systems.
	\newblock In {\em Lectures on probability theory. Ecole d'\'et\'e de
		probabilit\'es de Saint-Flour XXIII - 1993. Lectures given at the summer
		school in Saint- Flour, France, August 18-September 4, 1993}, pages 97--201.
	Berlin: Springer-Verlag, 1995.
	
	\bibitem[GS23]{gobron_saada_couplings_2023}
	Thierry Gobron and Ellen Saada.
	\newblock Couplings and attractiveness for general exclusion processes.
	\newblock In {\em Couplings and attractiveness for general exclusion
		processes}, volume~38 of {\em Ensaios Mat.}, pages 263--313. Soc. Brasil.
	Mat., Rio de Janeiro, 2023.
	
	\bibitem[Har72]{harris_1972}
	T.~E. Harris.
	\newblock Nearest-neighbor {Markov} interaction processes on multidimensional
	lattices.
	\newblock {\em Adv. Math.}, 9:66--89, 1972.
	
	\bibitem[Har78]{harris_1978}
	T.~E. Harris.
	\newblock Additive set-valued {Markov} processes and graphical methods.
	\newblock {\em Ann. Probab.}, 6:355--378, 1978.
	
	\bibitem[KL99]{kipnis_scaling_1999}
	Claude Kipnis and Claudio Landim.
	\newblock {\em Scaling limits of interacting particle systems}, volume 320 of
	{\em Grundlehren der Mathematischen Wissenschaften}.
	\newblock Springer-Verlag, Berlin, 1999.
	
	\bibitem[Lan93]{landim_conservation_1993}
	C.~Landim.
	\newblock Conservation of local equilibrium for attractive particle systems on
	{${\bf Z}^d$}.
	\newblock {\em Ann. Probab.}, 21(4):1782--1808, 1993.
	
	\bibitem[Lig99]{liggett_1999}
	Thomas~M. Liggett.
	\newblock {\em Stochastic interacting systems: contact, voter and exclusion
		processes}, volume 324 of {\em Grundlehren Math. Wiss.}
	\newblock Berlin: Springer, 1999.
	
	\bibitem[Lig05]{liggett_interacting_2005}
	Thomas~M. Liggett.
	\newblock {\em Interacting particle systems}.
	\newblock Classics in Mathematics. Springer-Verlag, Berlin, 2005.
	\newblock Reprint of the 1985 original.
	
	\bibitem[PBMH92]{paczuski1992fluctuations}
	Maya Paczuski, Mustansir Barma, S.~N. Majumdar, and T.~Hwa.
	\newblock Fluctuations of a nonequililbrium interface.
	\newblock {\em Phys. Rev. Lett.}, 69(18):2735, 1992.
	
	\bibitem[Pru04]{pruessner2004drift}
	Gunnar Pruessner.
	\newblock Drift causes anomalous exponents in growth processes.
	\newblock {\em Phys. Rev. Lett.}, 92(24):246101, 2004.
	
	\bibitem[Rez91]{rezakhanlou_hydrodynamic_1991}
	Fraydoun Rezakhanlou.
	\newblock Hydrodynamic limit for attractive particle systems on {${\bf Z}^d$}.
	\newblock {\em Comm. Math. Phys.}, 140(3):417--448, 1991.
	
	\bibitem[Ser99]{serre_systems_1999}
	Denis Serre.
	\newblock {\em Systems of conservation laws. 1}.
	\newblock Cambridge University Press, Cambridge, 1999.
	\newblock Hyperbolicity, entropies, shock waves, Translated from the 1996
	French original by I. N. Sneddon.
	
	\bibitem[Spo14]{spohn_nonlinear_2014}
	Herbert Spohn.
	\newblock Nonlinear fluctuating hydrodynamics for anharmonic chains.
	\newblock {\em J. Stat. Phys.}, 154(5):1191--1227, 2014.
	
	\bibitem[SW23]{shapira_wiese_anchored_2023}
	Assaf Shapira and Kay~J\"org Wiese.
	\newblock Anchored advected interfaces, {O}slo model, and roughness at
	depinning.
	\newblock {\em J. Stat. Mech. Theory Exp.}, (6):Paper No. 063202, 25, 2023.
	

	
\end{thebibliography}
\end{document}